\documentclass[12pt]{article}
\usepackage{amscd,amsfonts,amssymb,amsmath,amsthm,latexsym,mathrsfs,graphicx}
\usepackage[all]{xy}

\newcommand{\alg}{\text{alg}}

\newcommand{\bbF}{\mathbb{F}}

\newcommand{\bbQ}{\mathbb{Q}}

\newcommand{\bbQp}{{\mathbb{Q}_p}}

\newcommand{\bbZ}{\mathbb{Z}}

\newcommand{\bfB}{\mathbf{B}}

\newcommand{\bfD}{\mathbf{D}}

\newcommand{\bfF}{\mathbf{F}}

\newcommand{\bfM}{\mathbf{M}}

\newcommand{\bra}[1]{{\langle #1 \rangle}}
\newcommand{\bs}{\backslash}

\newcommand{\calO}{\mathcal{O}}

\newcommand{\cn}{\colon}
\newcommand{\crys}{\text{crys}}
\newcommand{\cycl}{\text{cycl}}
\newcommand{\dif}{\text{dif}}
\newcommand{\dR}{\text{dR}}

\newcommand{\ep}{\epsilon}
\newcommand{\et}{\text{\'et}}

\DeclareMathOperator{\Ext}{Ext}

\DeclareMathOperator{\Fil}{Fil}

\newcommand{\fkm}{\mathfrak{m}}

\DeclareMathOperator{\Frac}{Frac}
\newcommand{\Frob}{\text{Frob}}
\newcommand{\Ga}{\Gamma}
\newcommand{\ga}{\gamma}
\DeclareMathOperator{\Gal}{Gal}
\DeclareMathOperator{\GL}{GL}
\DeclareMathOperator{\Gr}{Gr}

\DeclareMathOperator{\img}{img}
\newcommand{\inv}{^{-1}}

\newcommand{\la}{\lambda}

\newcommand{\Nekovar}{{Nekov\'a\v r}}

\DeclareMathOperator{\ord}{ord}

\newcommand{\ov}[1]{{\overline{#1}}}

\newcommand{\pst}{\text{pst}}
\DeclareMathOperator{\rank}{rank}

\newcommand{\rig}{\text{rig}}
\newcommand{\rlim}{\varinjlim}
\newcommand{\scrB}{\mathscr{B}}

\newcommand{\scrD}{\mathscr{D}}
\newcommand{\scrE}{\mathscr{E}}

\newcommand{\scrT}{\mathscr{T}}
\newcommand{\scrU}{\mathscr{U}}
\newcommand{\scrV}{\mathscr{V}}
\newcommand{\scrW}{\mathscr{W}}
\newcommand{\scrX}{\mathscr{X}}

\DeclareMathOperator{\Spf}{Spf}
\DeclareMathOperator{\Spm}{Spm}
\newcommand{\st}{\text{st}}

\newcommand{\tord}{{\nabla\text{ord}}}
\newcommand{\triv}{\text{triv}}

\newcommand{\ul}[1]{{\underline{#1}}}
\newcommand{\unr}{\text{unr}}

\newcommand{\vphi}{\varphi}

\newcommand{\wh}[1]{{\widehat{#1}}}
\newcommand{\wt}[1]{{\widetilde{#1}}}

\newcommand{\thesisfinalsentence}{}

\newtheorem{thm}{Theorem}[section]
\newtheorem{ppn}[thm]{Proposition}
\newtheorem{lem}[thm]{Lemma}
\newtheorem{cor}[thm]{Corollary}
\newtheorem{conj}[thm]{Conjecture}

\theoremstyle{definition}
\newtheorem{exmp}[thm]{Example}

\theoremstyle{remark}
\newtheorem{rem}[thm]{Remark}
\newtheorem{altdefn}[thm]{Alternate definition}

\author{Jonathan Pottharst}
\title{Triangulordinary Selmer Groups}

\begin{document}

\maketitle

\abstract{Let $p$ be a prime number, and let $K$ be a $p$-adic local
  field.  We study a class of semistable $p$-adic Galois
  representations of $K$, which we call {\it triangulordinary} because
  it includes the ordinary ones yet allows non-\'etale behavior in the
  associated $(\vphi,\Ga_K)$-modules over the Robba ring.  Our main
  result provides a description of the Bloch--Kato local condition of
  such representations.  We also propose a program, using variational
  techniques, that would give a definition of the Selmer group along
  the eigencurve of Coleman--Mazur, including notably its nonordinary
  locus.}

\tableofcontents

\section{Introduction}

In his seminal work \cite{G1}, Greenberg laid out a conjectural
Iwasawa theory for a motive $M$ at an {\it ordinary} prime $p$.  His
ordinary hypothesis had the effect of drastically simplifying the
$p$-adic Hodge theory of $M$, while on the other hand being expected
to hold for a dense set of primes $p$.

Although our knowledge began to improve immediately after the time of
Greenberg's work, we learned that Iwasawa theory is, in comparison,
very complicated at nonordinary primes.  For example, Bloch--Kato
found in \cite{BK} the right definition of the general Selmer group,
and in \cite{PR} Perrin-Riou $p$-adically interpolated the Bloch--Kato
dual exponential map, providing a close link between Euler systems
(which bound Selmer groups) and $p$-adic $L$-functions.  While these
developments require no ordinary hypothesis, they rely heavily on
difficult crystalline techniques, and do not lead to a convenient
statement of the main conjecture punctually, let alone variationally.

Much more recently, there has been a major shift in the methods
underlying $p$-adic Hodge theory.  The work of many people has shown
that, essentially, all the important information attached to a
$p$-adic representation $V$ of the absolute Galois group $G_K$ of a
$p$-adic local field $K$ can be read rather easily from its
$(\vphi,\Ga_K)$-module, an invariant originally associated to $V$ by
Fontaine in \cite{F}, and subsequently refined by several authors.
(See \S\ref{sect-phigamma} for numerous details and references.)
Notably, the $(\vphi,\Ga_K)$-module of $V$ over the Robba ring may be
dissected into subquotients in ways that are not readily visible on
the level of the $p$-adic representation $V$ itself.  This was first
harnessed by Colmez, who called $V$ {\it trianguline} if its
$(\vphi,\Ga_K)$-module is a successive extension of $1$-dimensional
objects, and the latter notion has played a crucial role in our
burgeoning understanding of the $p$-adic local Langlands
correspondence for $\GL_2(\bbQ_p)$.

In this paper we use $(\vphi,\Ga_K)$-modules to give a natural
weakening of Greenberg's ordinary hypothesis.  We identify those
representations whose $(\vphi,\Ga_K)$-module is the same as that of an
ordinary representation (except possibly as regards the $\vphi$-slopes
of its ordinary filtration), and call them {\it triangulordinary}.  We
show how the $\vphi$-slopes are only rarely ever used when analyzing
the $p$-adic Hodge theory of such $V$.

We present two pieces of evidence that our hypothesis is natural and
timely.  First, as our main result we show that the natural analogue
of Greenberg's Selmer groups coincide with those defined by
Bloch--Kato.  This generalizes a result of Flach (see \cite[Lemma
2]{Fl}), which was proved using Poitou--Tate local duality and
Euler--Poincar\'e characteristic computations, to the case of
arbitrary perfect residue field.  Second, we propose a variational
program to extend our theory to {\it define} the Selmer module of the
universal finite-slope eigenform over a dense open subset of the
Coleman--Mazur eigencurve (and the eigensurface obtained from it by
cyclotomic twisting); such a definition has been hitherto unknown.
Our program would encompass results of Kisin, which provide Selmer
groups for all the individual overconvergent eigenforms in the family.

After the writing of this article, we found that many of our technical
results appear in \cite{BC2}.  The works have slightly different aims,
so let us briefly note how they differ.  First, throughout \cite{BC2}
one has $K=\bbQ_p$, so that, in particular, $\vphi$ is a linear
operator with well-defined eigenvalues; our theory does not even
assume that $K/\bbQ_p$ is finite, which is necessary for all
crystalline representations to be trianguline.  As concerns Selmer
groups, they only explicitly treat those associated to {\it adjoint}
representations, by measuring when trianguline deformations are
crystalline.  Our work is valid even when there is no
deformation-theoretic interpretation available.  In any case, their
methods can easily show that $H^1_\tord \subseteq H^1_f$; we explain
when equality holds.

This work would not even have been attempted, were it not for the
influence of many people.  We owe particular thanks to Laurent Berger
and Kiran Kedlaya for introducing us to this subject, and for their
patience in explaining its ideas to us.  Similarly, we would like to
thank the organizers of the 2005 ``Atelier sur les Repr\'esentations
$p$-adiques'' at CRM in Montreal, as well as the 2006 ``Special
Semester on Eigenvarieties'' at Harvard---our experiences there
incited us to take up a serious study of the ideas required to write
this article.  We thank Barry Mazur for his enthusiasm and
encouragement throughout this project.  We are indebted to Ruochuan
Liu and Ga\"etan Chenevier for extremely helpful conversations.  Jan
\Nekovar\ arranged for our stay in Paris, during which time much of
this work was hammered out.  Finally, we heartily thank the NSF for
its support through the MSGRFP, under which all this work was
completed, and l'Institut de Math\'ematiques de Jussieu for its
hospitality.

Let us conclude by describing the contents of the paper.  In the
following section, we gather in one place the facts about
$(\vphi,\Ga_K)$-modules, Galois cohomology, and $p$-adic Hodge theory
that will be required in the sequel.  Our aim is to provide a precise
resum\'e and guide to the literature.  In \S\ref{sect-local} we
present our results concerning individual Galois representations.
Here the reader will find the definition of triangulordinary
representations and proofs of their basic properties, including the
comparison of Selmer local conditions.  The section concludes by
describing the relationship to the notions of ordinary and
trianguline, and discussing examples arising in nature, including
abelian varieties and modular forms.  In the \S\ref{sect-global}, we
propose a program to define Selmer modules for general variations of
$p$-adic Galois representations, and show how this would apply to the
eigencurve and overconvergent $p$-adic modular forms.
\thesisfinalsentence

\section{Review of $(\vphi,\Ga_K)$-modules}\label{sect-phigamma}

For this entire section, we fix a complete, discretely valued field
$K$ of characteristic $0$, supposed to have a residue field $k$ that
is perfect of characteristic $p > 0$.  Choose once and for all an
algebraic closure $\ov{K}$ of $K$ and set $G_K = \Gal(\ov{K}/K)$.  Our
goal in this section is to review the relevant theory of
$(\vphi,\Ga_K)$-modules, which provide a means of describing
continuous $p$-adic representations of $G_K$ and their associated
invariants.

\subsection{Definitions of many rings}\label{sect-phigamma-rings}

In terms of our fixed $K$, we define a dizzying list of objects.  Our
notation most closely follows that of Colmez; in particular, our $r$
varies inversely with Berger's.  For any field $E$, write $E_n =
E(\mu_{p^n})$ for $n \leq \infty$.  If $E$ carries a valuation, write
$\calO_E$ for its ring of integers.

{\it Fields.}  Let $F = \Frac W(k)$ be the fraction field of the Witt
vectors of $k$.  Then $F$ embeds canonically into $K$ as its maximal
absolutely unramified subfield, and $K/F$ is a finite, totally
ramified extension.  If $k'$ denotes the residue field of $K_\infty$,
which is finite over $k$, then we define $F' = \Frac W(k')$.  Then
$F'$ is the maximal unramified extension of $F$ in
$K_\infty$\footnote{One can have $F \subsetneq F'$ and $F' \nsubseteq
K$!  Take, for example, $p=3$ and $K=\bbQ_3(\sqrt{3})$.}, and $K' =
K.F'$ is the maximal unramified extension of $K$ in $K_\infty$, so
that $K_\infty/K'$ is totally ramified.  Observe that, since $K'
\subseteq K_\infty$, for $n \gg 0$ one has $K' \subseteq K_n$ and
hence $K'_n = K_n$ and $K'_\infty = K_\infty$, and therefore $(K')' =
K'$.

We set $H_K = \Gal(\ov{K}/K_\infty)$ and $\Ga_K = \Gal(K_\infty/K)$.
The latter group is rather simple.  If $E$ is any field of
characteristic not equal to $p$, then the action of $\Gal(\ov{E}/E)$
on $\mu_{p^\infty}(E)$ is described by a uniquely determined character
$\chi_\cycl \cn \Gal(\ov{E}/E) \to \bbZ_p^\times$, called the
cyclotomic character.  The fundamental theorem of Galois theory in
this case says that $\chi_\cycl$ identifies $\Gal(E_\infty/E)$ with a
closed subgroup of $\bbZ_p^\times$.  In the case at hand, using the
fact that $K$ is discretely valued, one finds that $K_\infty/K$ is
infinite, so that $\chi_\cycl$ identifies $\Ga_K$ with an open
subgroup of $\bbZ_p^\times$, which by force must be procyclic or
$\{\pm1\} \times (\text{procyclic})$.  One has $H_{K'} = H_K$ by our
earlier remarks, and $\Ga_{K'}$ has finite index in $\Ga_K$.
Moreover, one has
\[
\Ga_K/\Ga_{K'} = \Gal(K'/K) = \Gal(F'/F) = \Gal(k'/k).
\]
We have the following diagram, where (for readability) $\star =
\Ga_K/\Ga_{K'}$, ``ur'' means unramified, and ``tr'' means totally
ramified.
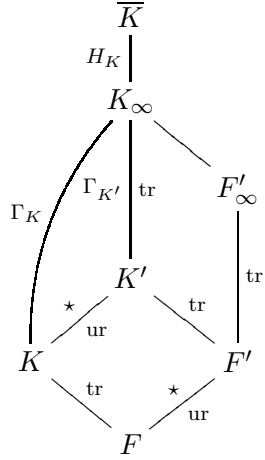
\[\xymatrix@=1.45pc{
& {\ov{K}} \\
& K_\infty \ar@{-}[u]^{H_K} \\
& & F'_\infty \ar@{-}[ul] \\
& K' \ar@{-}[uu]^{\Ga_{K'}}_{\text{tr}} \\
K \ar@{-}[ur]^\star_{\text{ur}} \ar@/^1pc/@{-}[uuur]^{\Ga_K}
& & F' \ar@{-}[ul]_{\text{tr}} \ar@{-}[uu]_{\text{tr}} \\
& F \ar@{-}[ul]_{\text{tr}} \ar@{-}[ur]^\star_{\text{ur}}
}\]

I would like to point out to the novice that dividing up $G_K$ into
$H_K$ and $\Ga_K$ is not traditional.  Classically, one divides up
$G_K$ into $I_K$ and $G_k$, where $I_K \subseteq G_K$ is the inertia
subgroup and $G_k = G_K/I_k$ is the absolute Galois group of $k$.
(Note that we have a canonical algebraic closure of $k$, namely the
residue field $\ov{k}$ of $\ov{K}$.)  In fact, it ends up not being
very hard to uncover traditional un/ramification information when
using $H_K$ and $\Ga_K$ instead, so this method is much more powerful,
at least in the setting of $p$-adic representations.

{\it Robba rings.}  There are three main variants of
$(\vphi,\Ga_K)$-modules, but we will only need the variety that live
over the Robba ring, so we now make a beeline for these.  All we
really need is that the ``field of norms'' construction allows one to
make a certain choice of an indeterminate $\pi_K$, and associates to
$K$ a constant $e_K\ (= \ord_{\wt{E}^+}(\ov{\pi}_K)) > 0$.  When
$K=F$, there is a {\it canonical} uniformizer which is written $\pi$,
and one can calculate that $e_F = p/(p-1)$.

Berger's Robba ring $B^\dag_{\rig,K}$ is defined to be the union of
the rings $B^{\dag,r}_{\rig,K}$ for $r>0$.  The latter are defined by
\[
B^{\dag,r}_{\rig,K} = \Bigg\{ f(\pi_K) = \sum_{n \in \bbZ} a_n \pi_K^n \
\Bigg|\ 
\begin{array}{cc}
a_n \in F',\\
f(X) \text{ convergent for } 0 < \ord_p(X) < r/e_K
\end{array} \Bigg\}.
\]
Although all these rings are non-Noetherian, they are not too
unpleasant.  For example, the rings $B^{\dag,r}_{\rig,K}$ are B\'ezout
domains: they admit a theory of principal divisors, and they have a
reasonable theory of finite free modules.  See \cite[\S4.2]{B1} for
details.

If $L$ is another CDVF with perfect residue field, with $K$
continuously embedded into it, there is a canonical embedding
$B^{\dag,r}_{\rig,K} \hookrightarrow B^{\dag,r}_{\rig,L}$ for $r$
sufficiently small.  More specifically, one can arrange for $\pi_L$ to
satisfy an Eisenstein polynomial over a subring of $B^\dag_{\rig,K}
\otimes_{F'} F_L'$ with respect to a suitable $\pi_K$-adic valuation.
(The term $F_L'$ is the maximal absolutely unramified subfield of
$L_\infty$, analogous to $F'$.)  The constants $e_K$ and $e_L$ are
normalized so that the growth conditions on power series coincide.
When $L/K$ is finite, we see that the $B^{\dag,r}_{\rig,K} \subseteq
B^{\dag,r}_{\rig,L}$ (for $r$ sufficiently small), and hence also
$B^\dag_{\rig,K} \subseteq B^\dag_{\rig,L}$, are finite ring
extensions.

A more delicate construction of these rings (as in \cite{B1}) endows
them with natural, commuting ring-endomorphism actions of $\Ga_K$ and
an operator $\vphi$.  One knows that $\vphi$ acts by Witt
functoriality on $a_n \in F'$, and $\Ga_K$ acts on $a_n$ through its
quotient $\Ga_K/\Ga_{K'} = \Gal(F'/F)$.  The action on $\pi_K$ is
generally not explicitly given (especially since there is some choice
in $\pi_K$), except when $K=F$, in which case $\vphi(\pi) =
(1+\pi)^p-1$ and $\ga \in \Ga_K$ obeys $\ga(\pi) =
(1+\pi)^{\chi_\cycl(\ga)}-1$.  The embeddings $B^\dag_{\rig,K}
\subseteq B^\dag_{\rig,L}$ are $\vphi$- and $\Ga_L$-equivariant
(considering $\Ga_L \hookrightarrow \Ga_K$).

Finally, we point out that the series $\log(1+\pi) = \sum_{n \geq 1}
\frac{(-1)^{n-1}}{n}\pi^n$ converges in $B^{\dag,r}_{\rig,\bbQp}$ for
every $r>0$, and we call its limit $t$.  By means of the above
embedding process, $t$ is an element of every $B^\dag_{\rig,K}$.  One
has $\vphi(t)=pt$ and $\ga(t) = \chi_\cycl(\ga)t$ for all $\ga \in
\Ga_K$.

\subsection{$\vphi$- and $(\vphi,\Ga_K)$-modules over the Robba
  ring}

Since many important facts about $(\vphi,\Ga_K)$-modules arise from
their underlying $\vphi$-modules, we first recall general properties
of $\vphi$-modules over the Robba ring.

Suppose $B$ is a ring equipped with a ring endomorphism $\vphi$.  A
$\vphi$-module over $B$ is a free, finite rank $B$-module $D$ equipped
with a semilinear action of $\vphi$, satisfying the nondegeneracy
condition that $\vphi(D)$ span $D$ over $B$.  The adjective
``semilinear'' indicates that one has $\vphi(bd) = \vphi(b)\vphi(d)$
for $b \in B$ and $d \in D$, rather than $\vphi(bd) = b\vphi(d)$.  We
write $\bfM(\vphi)_{/B}$ for the category of $\vphi$-modules.  Unless
otherwise specified, we understand that $B = B^\dag_{\rig,K}$.

It is worth noting that, in general, $\vphi(B^{\dag,r}_{\rig,K})
\nsubseteq B^{\dag,r}_{\rig,K}$, but instead
$\vphi(B^{\dag,r}_{\rig,K}) \subseteq B^{\dag,r/p}_{\rig,K}$ (the
latter of which contains $B^{\dag,r}_{\rig,K}$).  With this in mind,
it does not make sense to define a $\vphi$-module over
$B^{\dag,r}_{\rig,K}$.  The best we can (and will) ask for is a basis
of $D$ with respect to which the matrix for $\vphi$ lies in
$B^{\dag,r}_{\rig,K}$.  In this regard, there is the following crucial
lemma of Cherbonnier.

\begin{lem}[{\cite[Th\'eor\`eme I.3.3]{B2}}] For any $\vphi$-module $D$
and any $r$ sufficiently small, say $r < r(D)$, there exits a unique
$B^{\dag,r}_{\rig,K}$-lattice $D^r \subset D$ such that
$B^{\dag,r/p}_{\rig,K} \cdot \vphi(D^r)$ contains a basis of $D^r$.
One has
\[
B^{\dag,s}_{\rig,K} \otimes_{B^{\dag,r}_{\rig,K}} D^r
\stackrel{\sim}{\to} B^{\dag,s}_{\rig,K} \cdot D^r = D^s \subset D
\]
for $0 < s \leq r < r(D)$.
\end{lem}

We will make heavy use of the slope theory for $\vphi$-modules.
Currently the best reference for this material is \cite{K}, whose
proof easily specializes to give the classical Dieudonn\'e--Manin
theorem.

\begin{thm}[\cite{K}]
There is an over-ring $\wt{B}^\dag_\rig$ of $B^\dag_{\rig,K}
\otimes_{F'} \wh{F^\unr}$, with an extension of the operator $\vphi$
to it, such that the following claims hold.
\begin{enumerate}
\item For every $\vphi$-module $D$ over $\wt{B}^\dag_\rig$, there is a
  finite extension $L$ of $\wh{F^\unr}$ such that $D
  \otimes_\wh{F^\unr} L$ admits a basis of $\vphi$-eigenvectors, with
  eigenvalues in $L$.  The valuations of the eigenvalues, with
  multiplicity, are uniquely determined by $D$.  (Call them the {\em
  slopes} of $D$.)
\item There exists a unique filtration $\Fil^* \subseteq D$
  obeying the conditions that:
  \begin{itemize}
  \item each of the $\Gr^n$ has only one slope $s_n$ (but perhaps with
        multiplicity), and
  \item $s_1 < s_2 < \cdots < s_\ell$.
  \end{itemize}
\item If $D$ descends to $B^\dag_{\rig,K}$, then $\Fil^*$ descends
  with it uniquely.
\end{enumerate}
\end{thm}

One calls a $\vphi$-module {\it \'etale} if its only slope is $0$, and
denotes by $\bfM^\et(\vphi)_{/B^\dag_{\rig,K}} \subset
\bfM(\vphi)_{/B^\dag_{\rig,K}}$ the full subcategory of these.

A $(\vphi,\Ga_K)$-module (over $B^\dag_{\rig,K}$) is a $\vphi$-module
$D$ equipped with a semilinear action of $\Ga_K$ that is continuous
for varying $\ga \in \Ga_K$, and commutes with $\vphi$.  The category
of these is denoted by $\bfM(\vphi,\Ga_K)_{/B^\dag_{\rig,K}}$.  As a
consequence of the uniqueness statements in the above theorems about
$\vphi$-modules, one finds that $\Ga_K$ stabilizes the lattices $D^r$
and the slope filtration.  A $(\vphi,\Ga_K)$-module is called \'etale
if its underlying $\vphi$-module is; the category of these is written
$\bfM^\et(\vphi,\Ga_K)_{/B^\dag_{\rig,K}}$.

The following theorem, which combines work of many people, is the
reason that $(\vphi,\Ga_K)$-modules are important in the study of
$p$-adic Galois representations.  Let ${\bf Rep}_\bbQp(G_K)$ denote
the category of finite-dimensional $\bbQp$-vector spaces equipped with
a continuous, linear action of $G_K$.

\begin{thm}\label{thm-phigamma-equiv}
There is a canonical fully faithful embedding
\[
\bfD^\dag_\rig \cn {\bf Rep}_\bbQp(G_K) \hookrightarrow
\bfM(\vphi,\Ga_K)_{/B^\dag_{\rig,K}},
\]
whose essential image is $\bfM^\et(\vphi,\Ga_K)_{/B^\dag_{\rig,K}}$.
\end{thm}
\begin{proof}
The equivalence of Galois representations with \'etale
$(\vphi,\Ga_K)$-modules over the fraction field $B_K$ of the Cohen
ring of the field of norms is proved in \cite{F}.  The overconvergence
of such an object, which means that it can be uniquely defined over
$B^\dag_K \subset B_K$, is proved in \cite{CC}.  The equivalence of an
object over $B^\dag_K$ being \'etale over $B_K$ and \'etale over
$B^\dag_{\rig,K}$, as well as the unique descent of an \'etale object
over $B^\dag_{\rig,K}$ to $B^\dag_K$, follows from the slope theory of
\cite{K}.
\end{proof}

Let $L$ be another CDVF with perfect residue field, with $K$
continuously embedded into it, and let $D$ be a
$B^\dag_{\rig,K}$-module.  We use the shorthand
\[
D_L := D \otimes_{B^\dag_{\rig,K}} B^\dag_{\rig,L}
\]
throughout this article.  If $D$ has a $\vphi$-action, then so does
$D_L$.  If $D$ has a $\Ga_K$-action, then $D_L$ has a $\Ga_L$-action.
Suppose $\ov{L}$ is an algebraic closure of $L$ containing $\ov{K}$,
and $G_L = \Gal(\ov{L}/L)$.  Then for all $V \in {\bf
Rep}_\bbQp(G_K)$, one has $\bfD^\dag_\rig(V|_{G_L}) =
\bfD^\dag_\rig(V)_L$.

Can one recover the invariants of $V$ from $\bfD^\dag_\rig(V)$?
Indeed, and quite simply, as we recall in the next three sections.

\subsection{Galois cohomology}
\label{sect-phigamma-cohom}

Here we recall some results of Liu in \cite{L}, which are variants of
those of Herr, that allow one to recover the Galois cohomology of $V$
from $D = \bfD^\dag_\rig(V)$.

Recall that $\Ga_K \hookrightarrow \bbZ_p^\times$, and hence is
procyclic, except when $p=2$ and $-1 \in \img \Ga_K$, and in this case
$\Ga_K/\{\pm1\}$ is procyclic.  If $\Ga_K$ is procyclic, we set
$\Delta_K = \{1\}$, and otherwise we set $\Delta_K = \{\pm1\}$.  We
let $\ga \in \Ga_K/\Delta_K$ denote a topological generator.

In his thesis, Herr associated to $D$ the complex
\[
C^\bullet(D) = C_{\vphi,\ga}^\bullet(D) \cn 0 \to D^{\Delta_K}
\xrightarrow{(\vphi-1,\ga-1)} D^{\Delta_K} \oplus D^{\Delta_K}
\xrightarrow{(1-\ga) \oplus (\vphi-1)} D^{\Delta_K} \to 0,
\]
concentrated in degrees $[0,2]$.  One can easily show that the
cohomology groups
\[
H^i(D) = H^i(C_{\vphi,\ga}^\bullet(D))
\]
are independent of $\ga$, and moreover that they are canonically
identified with the Yoneda groups:
\[
H^i(D) = \Ext^i_{\bfM(\vphi,\Ga_K)_{/\bfB^\dag_{\rig,K}}}({\bf 1},D),
\]
where ${\bf 1}$ denotes the unit object (i.e.\ $B^\dag_{\rig,K}$
itself as a $(\vphi,\Ga_K)$-module).

In the case where $D = \bfD^\dag_\rig(V)$ is \'etale, we recover
continuous Galois cohomology.  Namely, by \cite[Theorem 2.6]{L} there
is a canonical isomorphism of $\delta$-functors $H^*(G_K,V)
\stackrel{\sim}{\to} H^*(\bfD^\dag_\rig(V))$.  In degree $i=1$, this
says that Galois cohomology classes of $V$ are in a natural bijection
with extension classes of the $(\vphi,\Ga_K)$-module ${\bf 1}$ by $D$.
It is these extension classes that we will be measuring later in this
article.

Similar (yet simpler) statements can be made for $\Ga_K$-modules
without $\vphi$-action: define $C_\ga^\bullet(D) \cn 0 \to
D^{\Delta_K} \xrightarrow{\ga-1} D^{\Delta_K} \to 0$, concentrated in
degrees $[0,1]$.  Then the cohomology of this complex is independent
of $\ga$ and computes the Yoneda groups
$\Ext^*_{\bfM(\Ga_K)_{/\bfB^\dag_{\rig,K}}}({\bf 1},D)$, as well as
the continuous group cohomology $H^*(\Ga_K,D)$.

\subsection{de~Rham theory}

We now explain the link between $(\vphi,\Ga_K)$-modules and $p$-adic
Hodge theory, first exploited by Cherbonnier--Colmez, and later
extended to the Robba ring by Berger.

In his thesis, Berger constructed maps $\iota_n \cn
B^{\dag,p^{-n}}_{\rig,K} \to K_n[\![t]\!]$ for all $n \geq n(K)$.  (In
particular, one assumes $n$ is large enough so that $K'_n[\![t]\!] =
K_n[\![t]\!]$.)

There are two ways to describe $\iota_n$.  On the one hand, one first
proves that $f \in \wt{B}^\dag_\rig$ (as in the slope filtration
theorem) converges in Fontaine's ring $B_\dR^+$ if and only $f \in
\wt{B}^{\dag,1}_\rig$.  Write $\iota(f) = \img(f) \in B_\dR^+$.  One
next shows that $\vphi^{-n}(\wt{B}^{\dag,p^{-n}}_\rig) \subseteq
\wt{B}^{\dag,1}_\rig$, and the image of $f \in
B^{\dag,p^{-n}}_{\rig,K}$ under $\iota_n = \iota \circ \vphi^{-n}$
lies in $K'_n[\![t]\!]$.

On the other hand, there is the following geometric picture when
$K=F$.  An element of $B^{\dag,r}_{\rig,F}$ is a rigid analytic
function on an annulus around $\pi=0$.  One has $r \geq 1$ if and only
if $\ep^{(1)}-1$ lies in this annulus.  Since $t = t(\pi)$ vanishes to
order $1$ at every $\pi=\ep^{(n)}-1$, it serves as a uniformizing
parameter there.  The map $\iota$ corresponds to the operation of
taking the formal germ at $\pi=\ep^{(1)}-1$.  For $n \geq 0$, the
operator $\vphi^{-n}$ stretches the domains of functions towards the
center of the disk.  After hitting a function by $\vphi\inv$ enough
times (i.e., if $f \in B^{\dag,r}_{\rig,F}$ and we ensure $rp^n \geq
1$), the point $\ep^{(1)}-1$ lies in its domain and we may localize
there.  Another way of saying this is that $\iota_n$ performs
completion at $\ep^{(n)}-1$.

Hopefully, the above description motivates the following formulas.
When $K=F$, given $f = \sum_{k \in \bbZ} a_k\pi^k$ with $a_k \in F$,
we may calculate $\iota_n(f)$ by means of the following rules:
\[
\iota_n(a_k) = \vphi^{-n}(a_k) \qquad \text{and} \qquad \iota_n(\pi) =
\ep^{(n)}\exp(t/p^n)-1.
\]
When $K$ does not necessarily equal $F$, a general $f$ has the form
$\sum_k a_k\pi_K^k$ with $a_k \in F'$.  Then one still has
$\iota_n(a_k) = \vphi^{-n}(a_k)$, but $\iota_n(\pi_K)$ is not
generally explicit.

The $\iota_n$ are $\Ga_K$-equivariant, and for varying $n$ they fit
into the following diagram.
\[\xymatrix{
B^{\dag,r}_{\rig,K} \ar[r]^{\iota_n} \ar[d]_\vphi
& K_n[\![t]\!] \ar@{^{(}->}[d] \\
B^{\dag,r/p}_{\rig,K} \ar[r]^{\iota_{n+1}} & K_{n+1}[\![t]\!]
}\]

Now we come to a fundamental construction.  Given a $\vphi$-module $D$
over $B^\dag_{\rig,K}$, we associate to it $D_\dif^+ = D^r
\otimes_{B^{\dag,r}_{\rig,K},\iota_n} K_\infty[\![t]\!]$, and $D_\dif
= D_\dif^+[t\inv]$.  Using the $\vphi$-structure in an essential way,
one shows that this definition is independent of the choices of $r$
and $n$ satisfying $r < r(V)$, $n \geq n(K)$, and $p^nr \geq 1$.

If $D$ is actually a $(\vphi,\Ga_K)$-module, then $D_\dif^+$ and
$D_\dif$ admit $\Ga_K$-actions.  Thus, we are able to define $D_\dR^+
= (D_\dif^+)^{\Ga_K}$ and $D_\dR = (D_\dif)^{\Ga_K}$.  These are
$K$-vector spaces of dimension $\leq \rank D$, and they carry a
decreasing, separated, and exhaustive filtration induced by the
$t$-adic filtration on $K_\infty(\!(t)\!)$.  One says that $D$ is {\it
de~Rham} (resp.\ {\it +de~Rham}) if $\dim_K D_\dR = \rank D$ (resp.\
$\dim_K D_\dR^+ = \rank D$), and denotes by
$\bfM^\dR(\vphi,\Ga_K)_{/B^\dag_{\rig,K}} \subset
\bfM(\vphi,\Ga_K)_{/B^\dag_{\rig,K}}$ the full subcategory of de~Rham
objects.

\begin{thm}[{\cite[\S5.3\footnote{Comparing \cite[Proposition 5.7
  and its proof]{B1} with [{\it loc.~cit.}, Corollaire 5.8] makes
  clear a slight typo in the statement of the proposition; it is the
  corrected form of the proposition that we use here.}]{B1}}] There
  exist functorial identifications respecting filtrations,
  $\bfD_\dR^+(V) = (\bfD^\dag_\rig(V))_\dR^+$ and $\bfD_\dR(V) =
  (\bfD^\dag_\rig(V))_\dR$.
\end{thm}

\begin{cor}
A representation $V$ is de~Rham (resp.\ +de~Rham) if and only if its
$(\vphi,\Ga_K)$-module is trivialized as semilinear $\Ga_K$-module
upon base change to $K_\infty(\!(t)\!)$ (resp.\ $K_\infty[\![t]\!]$).
\end{cor}

\begin{rem}\label{rem-moral}
Therefore, morally, only the {\it existence} of a $\vphi$-structure is
needed in order to construct $D_\dif^{(+)}$, and the property of $D$
being (+)de~Rham is predominantly a condition on the $\Ga_K$-action on
$D$.  This observation is the basis of our entire method; see
Proposition \ref{ppn-phi-irr} below for a precise statement.
\end{rem}

\subsection{$p$-adic monodromy}
\label{sect-phigamma-monodromy}

We will require the following results at a crucial point of the proof
of our main theorem, as well as to gain a more down-to-earth picture
of its content.

Given a $(\vphi,\Ga_K)$-module $D$, we write for brevity
\[
D[t\inv] = D \otimes_{B^\dag_{\rig,K}} B^\dag_{\rig,K}[t\inv]
\quad \text{and} \quad
D[\log(\pi),t\inv]
 = D \otimes_{B^\dag_{\rig,K}} B^\dag_{\rig,K}[\log(\pi),t\inv],
\]
where $t$ is the element of $B^\dag_{\rig,K}$ defined at the
conclusion of \S\ref{sect-phigamma-rings}, and where the element
$\log(\pi)$ is a free variable over $B^\dag_{\rig,K}$ equipped with
\[
\vphi(\log(\pi)) = p\log(\pi) + \log(\vphi(\pi)/\pi^p)
\quad \text{and} \quad
\ga(\log(\pi)) = \log(\pi) + \log(\ga(\pi)/\pi),
\]
the series $\log(\vphi(\pi)/\pi^p)$ and $\log(\ga(\pi)/\pi)$ being
convergent in $B^\dag_{\rig,\bbQ_p}$.  We associate to $D$ the modules
\[
D_\crys = D[t\inv]^{\Ga_K}
\quad \text{and} \quad
D_\st = D[\log(\pi),t\inv]^{\Ga_K}.
\]
These two modules are semilinear $\vphi$-modules over $F$, of
$F$-dimension $\leq \rank D$.  They are related via the so-called
monodromy operator $N$.  Namely, consider the unique
$B^\dag_{\rig,K}$-derivation $N \cn B^\dag_{\rig,K}[\log(\pi)] \to
B^\dag_{\rig,K}[\log(\pi)]$ satisfying $N(\log(\pi)) =
-\frac{p}{p-1}$.  It satisfies $N\vphi = p\vphi N$ and and commutes
with $\Ga_K$, and thus gives rise to an operator $N$ on $D_\st$ with
the property that $D_\crys = D_\st^{N=0}$.

We say that $D$ is {\it crystalline} (resp.\ {\it semistable}) if
$D_\crys$ (resp.\ $D_\st$) has the maximal $F$-dimension, namely
$\dim_F D_\crys = \rank D$ (resp.\ $\dim_F D_\st = \rank D$).  In
particular, $D$ is crystalline if and only if it is semistable and
$N=0$ on $D_\st$.  One can show that $D_\st \otimes_F K
\hookrightarrow D_\dR$, so that crystalline implies semistable and
semistable implies de~Rham.  We call $D$ {\it potentially crystalline}
(resp.\ {\it potentially semistable}) if there exists a finite
extension $L/K$ such that $D_L$ is crystalline (resp.\ semistable),
when considered as a $(\vphi,\Ga_L)$-module.  The following statement
is known as Berger's $p$-adic local monodromy theorem.

\begin{thm}[\cite{B1}]
Every de~Rham $(\vphi,\Ga_K)$-module is potentially semistable.
\end{thm}

The upshot of this theorem is that, whereas in the last section
$D_\dR$ was merely a filtered $K$-vector space, now we may equip it
with much more structure.  Given a de~Rham $D$, let $L/K$ be finite
Galois such that $D_L$ is semistable.  Then $(D_L)_\st$ is a
$(\vphi,N)$-module over the maximal absolutely unramified subfield
$F_L$ of $L$, and $(D_L)_\st \otimes_{F_L} L = (D_L)_\dR$ is a
filtered $L$-vector space.  Essentially because these data arise via
restriction from $K$, they are naturally equipped with a semilinear
action of $\Gal(L/K)$ that commutes with $\vphi$ and $N$ and preserves
the filtration.  Such an object is called a {\it filtered
$(\vphi,N,\Gal(L/K))$-module over $K$}.

Given two extensions $L_i$ and filtered
$(\vphi,N,\Gal(L_i/K))$-modules $D_i$ (for $i=1,2$), we consider them
equivalent if there exists an extension $L$ containing the $L_i$ such
that the $D_i$ tensored up to $L$ are isomorphic.  When we consider
objects only up to this equivalence, we call them {\it filtered
$(\vphi,N,G_K)$-modules}, thinking of the $G_K$-action as being
through an unspecified finite quotient that determines the field of
definition of the underlying vector spaces.  We point out that if $D$
becomes semistable over both $L_1$ and $L_2$, then $(D_{L_1})_\st$ and
$(D_{L_2})_\st$ are equivalent.  We call this equivalence class
$D_\pst$.  (To avoid set-theoretic issues, one simply deals with
filtered $(\vphi,N,G_K)$-modules whose underlying $\vphi$-module is a
vector space over $F^\unr$, and underlying filtered vector space has
coefficients in $\ov{K}$, with the assumption that $G_K$ acts
discretely.)  The category of filtered $(\vphi,N,G_K)$-modules is
denoted $\bfM\bfF(\vphi,N,G_K)$.

Summarizing the $p$-adic monodromy theorem in terms of the above
language, if $D$ is de~Rham then $D_\pst$ determines a filtered
$(\vphi,N,G_K)$-module over $K$.  Following Fontaine, a filtered
$(\vphi,N,G_K)$-module $D_\pst$ is called {\it (weakly) admissible}
if, roughly, its Newton and Hodge polygons have the same endpoints,
and all its $(\vphi,N)$-stable submodules satisfy ``Newton on or above
Hodge''.  (See \cite[{\S}I.1]{B2} for details.)

\begin{thm}[\cite{CF,B2}]
The functor $D \mapsto D_\pst$ is an equivalence of categories:
\[
\bfM^\dR(\vphi,\Ga_K)_{/B^\dag_{\rig,K}}
\stackrel{\sim}{\longrightarrow}
\bfM\bfF(\vphi,N,G_K).
\]
A de~Rham $(\vphi,\Ga_K)$-module $D$ is \'etale if and only if
$D_\pst$ is (weakly) admissible.  It is potentially crystalline if and
only if $N=0$ on $D_\pst$.  It is semistable if and only if $D_\pst$
can be realized as a filtered $(\vphi,N,\Gal(K/K))$-module.
\end{thm}

The two equivalent categories above are {\it not} abelian categories.
In the first category, the coimage of a map $D \to D'$ is the
set-theoretic image, while the image is the $t$-saturation of this set
(i.e.\ the elements $x \in D'$ such that some $t^nx$ lies in the
set-theoretic image).  In the second category, the coimage and image
have the same underlying $(\vphi,N,G_K)$-module, but different
filtrations.  The filtration on the coimage is induced by the
surjection from $D$, and the filtration on the image is induced from
the inclusion into $D'$.  Thus, $t$-saturated $(\vphi,\Ga_K)$-stable
$B^\dag_{\rig,K}$-submodules of $D$ are in a natural correspondence
with subspaces of $D_\pst$ that are stable under the
$(\vphi,N,G_K)$-actions (considered as being equipped with the
filtration induced from $D_\pst$).  Furthermore, one can show that a
$t$-saturated $B^\dag_{\rig,K}$-submodule is actually a direct
summand, provided that it is $(\vphi,\Ga_K)$-stable.

Moreover, the proof in \cite{B2} explains the following facts.  For
simplicity, assume that $D$ is crystalline.  Consider $D$ and $D_\crys
\otimes_F B^\dag_{\rig,K}$ as a $B^\dag_{\rig,K}$-lattices in
$D[t\inv]$.  As one passes from the first to the second, one invokes
multiples by various $t^n$ (with $n \in \bbZ$) in order to trivialize
the $\Ga_K$-action on some basis.  But multiplying by $t^n$ shifts
$\vphi$-slopes upwards by $n$, and thus, as we change lattices, the
$\vphi$-slopes get dragged to new values.  But the powers of $t$
involved, which determine the amount of dragging, more directly
determine the weights of the Hodge filtration on $D_\crys$.  So, there
is a close (but complicated!)  connection between the Hodge--Tate
weights on $D_\crys$ and the {\it difference} between the
$\vphi$-slopes on $D$ and the $\vphi$-slopes on $D_\crys$.  In the
case where $D$ is a trivial $\Ga_K$-module, i.e.\ $D \approx
(B^\dag_{\rig,K})^{\oplus d}$ as a $\Ga_K$-module, one clearly sees
that $D$ is crystalline and that there is no change of lattice, so the
$\vphi$-slopes on $D$ {\it coincide} with the $\vphi$-slopes on
$D_\crys$.

\section{Local theory}\label{sect-local}

In this section we define triangulordinary representations, and prove
that they have many amenable properties.  We go on to define Selmer
groups of representations that are triangulordinary at $p$, and give
examples.

We continue with the notations set forth in \S\ref{sect-phigamma}.

\subsection{Triangulordinary $(\vphi,\Ga_K)$-modules and
  Selmer groups}\label{sect-definitions}

When in doubt, $D$ refers to an object in
$\bfM(\vphi,\Ga_K)_{/B^\dag_{\rig,K}}$.  In this subsection we set $L
= \wh{K^\unr}$, and remind the reader of the meaning of $D_L$: the
compositum $\ov{L}:=\ov{K}.L$ is an algebraic closure of $L$, and $G_L
= \Gal(\ov{L}/L) \stackrel{\sim}{\to} I_K \subset G_K$, the inertia
subgroup.  Thus, for $V \in {\bf Rep}_\bbQp(G_K)$, one has
$\bfD^\dag_\rig(V)_L = \bfD^\dag_\rig(V|_{I_K})$.

We say that $D$ is {\it triangulordinary} if there exists a
decreasing, separated and exhaustive, $(\vphi,\Ga_K)$-stable
filtration $F^* \subseteq D$ by $B^\dag_{\rig,K}$-direct summands,
such that each $(\Gr_F^n)_L$ is $\Ga_L$-isomorphic to
$(t^nB^\dag_{\rig,L})^{\oplus d_n}$.  The $d_n$ are called the
multiplicities of the weights $n$ in $D$; for example, $D$ has weights
$\geq 1$ if and only if $D = F^1$.  For an equivalent definitions see
Corollary \ref{cor-crys-HT-n} and \S\ref{sect-compare}, and for
examples see \S\ref{sect-local-ex}. We do not give these immediately,
because discussing them rigorously requires several tools.

Given $D$ and a triangulordinary filtration $F^*$, we put
\begin{equation}\label{eqn-tord-selmer-cond}
H^1_\tord(D) = H^1_\tord(D;F^*) = \ker\left[ H^1(D) \to
  H^1((D/F^1)_L) \right],
\end{equation}
and call it the {\it triangulordinary local condition}.  It is
possible that $D$ be triangulordinary with respect to more than one
filtration (see \S\ref{sect-local-ex}), hence the need for the
``$F^*$'' in the notation.  However, we will usually have a fixed
filtration in mind, and therefore we will usually drop it from sight.

Here is the main theorem of this section.

\begin{thm}\label{thm-local-main}
Let $D$ be triangulordinary with filtration $F^*$.  Then the following
claims hold.
\begin{enumerate}
\item $D$ is de~Rham (and moreover +de~Rham if and only if $F^1 = 0$),
  and even semistable.
\item Suppose for all $n \leq 0$ that $n-1$ is not a $\vphi$-slope on
  $\Gr_F^n$.  Then $H^1_\tord(D;F^*)$ coincides with $H^1_{g+}(D)$,
  defined in \S\ref{sect-galois-descent}.
\end{enumerate}
\end{thm}

We prove this theorem in \S\ref{sect-local-proof}; the intervening
sections involve preparatory material.

For the remainder of this subsection, let us break with the running
notation.  Let $K/\bbQ$ be a finite extension, and let $S$ be a finite
set of places of $K$ containing all primes above $p$ and $\infty$.
Write $K_S$ for a maximal extension of $K$ unramified outside $S$, and
$G_{K,S} = \Gal(K_S/K)$.  Choose algebraic closures $\ov{K}_v$ of
$K_v$, and write $G_v = \Gal(\ov{K}_v/K)$, for each finite place $v
\in S$; choose embeddings $K_S \hookrightarrow \ov{K}_v$, which amount
to maps $G_v \to G_{K,S}$.  Also, for $v \in S$ with $v \nmid
p\infty$, denote by $I_v \subset G_v$ the inertia subgroup, and write
$G_{\bbF_v} = G_v/I_v$.

Let $V$ be a finite-dimensional $\bbQ_p$-vector space equipped with a
continuous, linear action of $G_{K,S}$.  Assume that, for each place
$v$ of $K$ with $v \mid p$, $\bfD^\dag_\rig(V|_{G_v})$ is equipped
with a triangulordinary filtration $F_v^*$.  We define the {\it
triangulordinary local conditions} as above: they are the respective
subgroups $H^1_\tord(K_v,V)$ corresponding, under the identifications
$H^1(K_v,V) \cong H^1(\bfD^\dag_\rig(V|_{G_{K_v}})$, to the subgroups
$H^1_\tord(\bfD^\dag_\rig(V|_{G_{K_v}}))$.  Then, following the
customary pattern, we define the {\it triangulordinary Selmer group}
to be
\begin{equation}\label{eqn-tord-selmer}
H^1_\tord(K,V) = \ker\left[ H^1(G_{K,S},V) \to
\bigoplus_{\substack{v \in S\\ v \nmid p\infty}}
  \frac{H^1(K_v,V)}{H^1(G_{\bbF_v},V^{I_v})}
\oplus
\bigoplus_{\substack{v \in S,\\ v \mid p}}
  \frac{H^1(K_v,V)}{H^1_\tord(K_v,V)} \right].
\end{equation}
After proving Theorem \ref{thm-local-main}, we will see how this
definition generalizes Selmer groups defined by Greenberg, and agrees
with those defined by Bloch--Kato.

\subsection{Galois descent}\label{sect-galois-descent}

In this section and the next we show that the de~Rham property is
rather flexible: it is easy to equate the validity of this property
for one $(\vphi,\Ga_K)$-module to its validity for another one.  The
instance in this section concerns the ability to discern that $D$ is
de~Rham (resp.\ crystalline), given that the restriction $D_L$ of $D$
to some possibly large overfield $L \supseteq K$ has the same
property.  I suspect that these facts are known to the experts, but I
give precise statements and proofs because they do not appear in the
literature in the generality of possibly non-\'etale
$(\vphi,\Ga_K)$-modules.

By a {\it complete unramified extension} $L$ of $K$, we mean the
$p$-adic completion of an unramified (but possibly infinite) algebraic
extension of $K$.  Using an appropriate variant of the Witt vectors
formalism, such fields $L$ lie in a natural bijection with algebraic
extensions of $k$, and hence to closed subgroups $H$ of $G_k$.  When
$H$ is normal in $G_K$, we call $L$ {\it normal} and set $\Gal(L/K) =
G_k/H$.  For example, the maximal complete unramified extension of $K$
is $L = \wh{K^\unr}$, with $\Gal(L/K) = G_k$.

By a {\it complete discretely valued extension} $L$ of $K$, we mean a
finite extension of a complete unramified extension.  These are the
same as CDVFs into which $K$ embeds continuously, such that the
embedding induces an algebraic extension of residue fields.  The
complete discretely valued extensions $L$ of $K$ are in a natural
bijection with closed subgroups $H$ of $G_K$ for which $H \cap I_K$
has finite index in $I_K$.  If $H$ is normal in $G_K$, we call $L$
{\it normal} and we set $\Gal(L/K) = G_K/H$.  The class of complete
discretely valued extensions is closed under finite composita, but
does not admit a maximal element.

\begin{ppn}\label{ppn-derham-descent}
Let $D \in \bfM(\vphi,\Ga_K)_{/B^\dag_{\rig,K}}$, and let $L$ be a
complete discretely valued extension of $K$.  Then $\dim_K D_\dR^{(+)}
= \dim_L (D_L)_\dR^{(+)}$.
\end{ppn}

\begin{rem}
In the down-to-earth terms of a Galois representation $V$, the
proposition says that the de~Rham periods of $V$ essentially only
depend on the action upon $V$ of an arbitrarily small finite index
subgroup of the inertia group $I_K \subseteq G_K$.
\end{rem}

\begin{proof}
It is clear that $D_\dR^{(+)} \otimes_K L \hookrightarrow
(D_L)_\dR^{(+)}$, which shows the inequality $\leq$.  This proof
consists of showing the reverse inequality.  Notice that, for any
particular $L$, and any complete discretely valued extension $L'/L$,
we know the inequality $\leq$ for $L'/L$, and if we know $\geq$ for
the extension $L'/K$, then we know as a result the inequality $\geq$
for $L/K$ and $L'/L$.  Therefore, it suffices to prove the proposition
with $L'$ in place of $L$, i.e.\ it never hurts to enlarge the $L$ in
question.  In particular, by passing to the (completed) normal
closure, we may assume that $L/K$ is normal.

We first establish the proposition in the case when $L/K$ is {\it
finite}.  The idea is to harness the ability to enlarge $L$ in order
to really only treat two cases: when $L$ and $K_\infty$ are linearly
disjoint over $K$, and when $L \subset K_\infty$.  Let $L_0 = L \cap
K_\infty$, so that $L$ and $K_\infty$ are linearly disjoint over $L_0$
(because $L$ and $K_\infty$ are both normal over $K$).

We treat the extension $L/L_0$ first.  We have
\[
(D_L)_\dif^{(+)} = (D_{L_0})_\dif^{(+)}
  \otimes_{L_{0,\infty}[\![t]\!]}  L_\infty[\![t]\!].
\]
Notice that $\Gal(L/L_0)$ acts only on the right hand factor of this
expression, and it commutes with the $\Ga_L$-action (since $L$ and
$K_\infty$ are linearly disjoint over $L_0$).  Thus, the
$\Gal(L/L_0)$-invariants of $(D_L)_\dif^{(+)}$ are the $\Ga_L =
\Ga_{L_0}$-module $(D_{L_0})_\dif^{(+)}$.  In other words,
$(D_{L_0})_\dR^{(+)} = ((D_L)_\dR^{(+)})^{\Gal(L/L_0)}$.  But
$(D_L)_\dR^{(+)}$ is an $L$-vector space of dimension $\rank D$
equipped with a continuous, semilinear action of $\Gal(L/L_0)$, and
Hilbert's Theorem 90 for finite, Galois field extensions implies that
all finite-dimensional, semilinear $\Gal(L/L_0)$-modules over $L$ are
trivial.  Thus, $(D_L)_\dR^{(+)}$ admits a basis of elements fixed
under $\Gal(L/L_0)$, which shows that $(D_{L_0})_\dR^{(+)}$ has the
same dimension as dimension $(D_L)_\dR^{(+)}$, and we have $\geq$ in
this case.

We may now assume that $L = L_0$, so that $L$ is contained in
$K_\infty$.  But now we have $(D_L)_\dif^{(+)} = D_\dif^{(+)}$ {\it as
$K_\infty[\![t]\!]$-modules}, and the $\Ga_L$-action on the left hand
side obtained by restricting the $\Ga_K$-action on the right hand
side.  Since $\Ga_K$ is abelian, its action on $(D_L)_\dif^{(+)}$
commutes with the $\Ga_L$-action and induces a semilinear $\Ga_K/\Ga_L
= \Gal(L/K)$-action over $L$ on $D_\dif^{(+)}$.  We again apply
Hilbert's Theorem 90 to deduce the desired inequality.

Now we turn to the infinite case, and make use of the proposition in
the finite case to simplify things.  Suppose we are given a tower $K
\subseteq L_0 \subseteq L$, with $L_0/K$ complete unramified and
$L/L_0$ finite.  Applying the finite case to $L/L_0$, we see that we
have equality for $L_0/K$ if and only if we have equality for $L/K$.
Thus we are reduced to proving the proposition when $L = L_0$, so that
$L/K$ is complete unramified.

Next consider the extension $K'/K$.  Since it is finite unramified,
the extension $(K'.L)/L$ is also finite unramified.  By the finite
case of the proposition, we have equality for the extension
$(K'.L)/L$, so we are reduced to the case where $L$ contains $K'$.
Considering the tower $K \subseteq K' \subseteq L$, we have equality
for $L/K$ if and only if we have equality for $L/K'$.  Thus we can
assume that $K = K'$ and $L/K$ is complete unramified.

Note that $K_\infty/K$ is now totally ramified, and hence also
linearly disjoint with (any algebraic subextension of) $L/K$.  This
implies the following facts.  First, $\Ga_L = \Ga_K$.  Moreover, the
actions of $\Ga_L$ and $\Gal(L/K)$ on $L_\infty$ commute with one
another.  Finally, we infer that $L_\infty/L$ is totally ramified, or,
equivalently, that $L' = L$ (the left hand side being the maximal
unramified extension of $L$ in $L_\infty$).

Consider the module $(D_L)_\dif^{(+)}$, which can be written as
$D_\dif^{(+)} \otimes_{K_\infty[\![t]\!]} L_\infty[\![t]\!]$.  Notice
that $\Gal(L/K)$ acts only on the right hand factor of this expression
for $(D_L)_\dif^{(+)}$, and it commutes with the $\Ga_L$-action.
Thus, the $\Gal(L/K)$-invariants of $(D_L)_\dif^{(+)}$ are the
$\Ga_L=\Ga_K$-module $D_\dif^{(+)}$.  In other words, $D_\dR^{(+)} =
((D_L)_\dR^{(+)})^{\Gal(L/K)}$.  So, $(D_L)_\dR^{(+)}$ is an
$L$-vector space with a continuous, semilinear action of $\Gal(L/K)$.
Invoking Hilbert's Theorem 90 for complete unramified extensions
(obtained by limits and d\'evissage from the traditional theorem for
$k_L/k$), we see that $(D_L)_\dR^{(+)}$ admits a basis of
$\Gal(L/K)$-invariants.  Thus we have equality for $L/K$, as was
desired.
\end{proof}

\begin{cor}\label{cor-derham-descent}
Let $D$ be a $(\vphi,\Ga_K)$-module, and let $L/K$ be an extension, as
in the preceding theorem.  Then $D$ is (+)de~Rham if and only if $D_L$
is (+)de~Rham (the latter considered as a $(\vphi,\Ga_L)$-module).
\end{cor}

\begin{proof}
By the theorem, $\dim_K D_\dR^{(+)} = \rank D$ if and only if $\dim_L
(D_L)_\dR^{(+)} = \rank D$.
\end{proof}

Our main use of the proposition is in the case of extension classes.
Let us describe how this occurs.  Because the cohomology groups
$H^*(D)$ defined in \ref{sect-phigamma-cohom} coincide with Yoneda
groups, to every $c \in H^1(D)$ there corresponds a class of
extensions
\[
0 \to D \to E_c \to {\bf 1} \to 0.
\]
Since the functor $D \mapsto D_\dif^{(+)}$ corresponds to changing the
base rings of finite free objects (over Bezout domains), it preserves
short exact sequences.  Therefore, we can hit the exact sequence above
with this functor to obtain an exact sequence of $\Ga_K$-modules over
$K_\infty[\![t]\!]$ or $K_\infty(\!(t)\!)$.  The we obtain thus a map
$H^1(D) \to H^1(\Ga_K,D_\dif^{(+)})$ by $[E_c] \mapsto
[(E_c)_\dif^{(+)}]$.

Also, if $L/K$ is an extension as in the proposition
above, then $[E_c] \mapsto [(E_c)_L]$ defines a map denoted $\alpha_L
\cn H^1(D) \to H^1(D_L)$.

The {\it Bloch--Kato ``g(+)'' local condition} is the subgroup of
$H^1(D)$ determined by
\[
H^1_{g(+)}(D) = \ker\left[ H^1(D) \to H^1(\Ga_K,D_\dif^{(+)}) \right].
\]
Proposition \ref{ppn-derham-descent} now gives the following descent
result for these subgroups.

\begin{cor}\label{cor-BK-descent}
For any $D$ and $L/K$ as in the statement of Proposition
\ref{ppn-derham-descent}, one has $H^1_{g(+)}(D) = \alpha_L\inv
H^1_{g(+)}(D_L)$, where $\alpha_L$ is defined above.

In particular, taking $L = \wh{K^\unr}$, we see that $H^1_\unr(D) :=
\ker \alpha_{\wh{K^\unr}} \subseteq H^1_{g(+)}(D)$.
\end{cor}

\begin{proof}
A $\Ga_K$-fixed vector of $(E_c)_\dR^{(+)}$ not belonging to the
subspace $D_\dR^{(+)}$ is the same thing as a $\Ga_K$-equivariant
splitting of the map $(E_c)_\dif^{(+)} \to {\bf 1}_\dif^{(+)}$.  Thus,
$(E_c)_\dif^{(+)}$ is $\Ga_K$-split if and only if $\dim_K
(E_c)_\dR^{(+)} = \dim_K D_\dR^{(+)} + 1$.  By the theorem, this holds
if and only if the corresponding claim holds with $K$, $D$, and $E_c$
replaced by $L$, $D_L$, and $(E_c)_L$, respectively.  Thus
$(E_c)_\dif^{(+)}$ is $\Ga_K$-split if and only if
$((E_c)_L)_\dif^{(+)}$ is $\Ga_L$-split.  In other words, we have $c
\in H^1_{g(+)}(D)$ if and only if $\alpha_L(c) \in H^1_{g(+)}(D_L)$,
as was desired.
\end{proof}

Suppose that $D$ is (+)de~Rham, and $c \in H^1(D)$.  We see from the
above proof that $(E_c)_\dif^{(+)}$ is $\Ga_K$-split if and only if
\[
\dim_K (E_c)_\dR^{(+)} = \dim_K D_\dR^{(+)} + 1 = \rank D + 1 = \rank
E,
\]
which occurs if and only if $E_c$ is itself (+)de~Rham.  Thus, in this
case, the Bloch--Kato ``g(+)'' local condition can be interpreted as
the subgroup of $H^1(D)$ determined by
\[
H^1_{g(+)}(D) = \{c \in H^1(D) \mid E_c \text{ is (+)de~Rham}\}.
\]

We also prove a descent result for the crystalline property.  Notice
that it includes the case of finite unramified extensions $L/K$.  For
a complete discretely valued extension $L/K$, we write $F_L$ for the
maximal absolutely unramified subfield of $L$, i.e.\ the field that
would be called $F$ if we replaced $K$ with $L$.

\begin{ppn}\label{ppn-crys-descent}
Let $D \in \bfM(\vphi,\Ga_K)_{/B^\dag_{\rig,K}}$, and let $L/K$ be a
complete unramified extension.  Then $\dim_F D_\crys = \dim_{F_L}
(D_L)_\crys$.  In particular, $D$ is crystalline if and only if $D_L$
is.
\end{ppn}

\begin{proof}
Since $D_\crys \otimes_F F_L \hookrightarrow (D_L)_\crys$, we show the
inequality $\geq$.

We may make reductions just as in the proof of Proposition
\ref{ppn-derham-descent}; namely, it suffices to assume that $L$ is
Galois over $K$ and contains $K'$, and to just treat independently the
cases where $K=K'$ and $L=K'$.

Assume $L=K'$.  Then since $(K')' = K'$ (see
\S\ref{sect-phigamma-rings}), one has $D_{K'} = D$ {\it as sets}, and
we see that
\[
(D_{K'})_\crys = D_{K'}[t\inv]^{\Ga_{K'}} = D[t\inv]^{\Ga_{K'}}
\]
is a semilinear $\Ga_K/\Ga_{K'} = \Gal(F'/F)$-module over $F'$
satisfying $(D_{K'})_\crys^{\Ga_K/\Ga_{K'}} = D_\crys$.  It must be
trivial, by Hilbert's Theorem 90, and hence $D_\crys$ has the desired
$F$-dimension.

Now assume $K=K'$.  Since $K_\infty/K$ is totally ramified, and $L/K$
is unramified, the two extensions are linearly disjoint.  In
particular, $\Ga_L = \Ga_K$ and $L' = L$.  We have
\[
D_L[t\inv] = D[t\inv] \otimes_{B^\dag_{\rig,K}} B^\dag_{\rig,L},
\]
and by linear disjointness the $\Gal(L/K)$-action on the right hand
factor commutes with the $\Ga_K$-action on the tensor product.  This,
combined with the fact that $\Ga_L=\Ga_K$, shows that
$(D_L)_\crys^{\Gal(L/K)} = D_\crys$.  We know that $(D_L)_\crys$ is a
semilinear $\Gal(L/K) = \Gal(F_L/F)$-module over $F_L$.  Applying
Hilbert's Theorem 90, it admits a basis of invariants, and hence
$D_\crys$ has the desired rank.
\end{proof}

\begin{cor}\label{cor-crys-HT-n}
Write $L = \wh{K^\unr}$.  For a $(\vphi,\Ga_K)$-module $D$, the
following claims are equivalent:
\begin{itemize}
\item $D$ is $\Ga_K$-isomorphic to $(t^nB^\dag_{\rig,K})^{\oplus d}$.
\item $D_L$ is $\Ga_L$-isomorphic to $(t^nB^\dag_{\rig,L})^{\oplus
  d}$.
\item $D$ is crystalline, and all its Hodge--Tate weights equal $n$.
\end{itemize}
\end{cor}

\begin{proof}
The first condition clearly implies the second.

Assuming the second condition, $D_L$ is crystalline (since
$(D_L)_\crys = (t^{-n}D_L)^{\Ga_L}$ is clearly large enough), whence
Proposition \ref{ppn-crys-descent} shows that $D$ is crystalline.  And
$\bfD_\dif^+(D) = (t^nK_\infty[\![t]\!])^{\oplus d}$, showing that the
Hodge--Tate weights are all $n$.

Now assume that the third condition holds.  The fact that $D$ is
$\Ga_K$-isomorphic to $(t^nB^\dag_{\rig,K})^{\oplus d}$ results
directly from a check of the construction $D_\pst \mapsto D$ given in
\cite[\S II]{B2}.
\end{proof}

This corollary allows us to restate the condition that a filtration
$F^* \subseteq D$ be triangulordinary.  The hypothesis becomes: each
$\Gr_F^n$ is crystalline, with all Hodge--Tate weights equal to $n$.

\subsection{Irrelevance of $\vphi$-structure}
\label{sect-local-phi}

Next we state and prove a precise version of Remark \ref{rem-moral}.

\begin{ppn}\label{ppn-phi-irr}
The $\Ga_K$-isomorphism class of $D_\dif^+$ does not depend on which
$\vphi$-structure $D$ is equipped with (although the existence of
$\vphi$ is necessary to define $D_\dif^+$).  The same claim holds for
$D_\dif$.
\end{ppn}

\begin{proof}
The construction of $D^r$ in the proof of \cite[Th\'eor\`eme
I.3.3]{B2} shows that, for {\it any} $B^\dag_{\rig,K}$-basis $e =
\{e_1,\ldots,e_d\}$, there exists $0 < r(e) < r(D)$ such that the
$\Ga_K$-action on $e$ is defined over $B^{\dag,r(e)}_{\rig,K}$ and if
$0 < r \leq r(e)$ then $D^r$ is $B^{\dag,r}_{\rig,K}$-spanned by $e$.
Moreover, all the modules $D^r \otimes_{B^{\dag,r}_{\rig,K},\iota_n}
K_\infty[\![t]\!]$ for $0 < r < r(D)$, $n \geq n(K)$, and $rp^n \geq
1$ are isomorphic as $\Ga_K$-modules.  Thus, if we consider $D_\dif^+$
as the $K_\infty[\![t]\!]$-span of $e$ with $\Ga_K$-action piped
through $\iota_n$, then the resulting isomorphism class stabilizes for
$n \gg 0$.  We simply take $n$ large enough to ensure this.  (In other
words, $\vphi$ only guarantees that the $\Ga_K$-structures are
equivalent and chooses the equivalence, but they are all equivalent no
matter which $\vphi$ is used to show it.)

The claim for $D_\dif$ follows from the claim for $D_\dif^+$ after
inverting $t$.
\end{proof}

\begin{cor}
If $D$ and $D'$ are two $(\vphi,\Ga_K)$-modules, and $D \cong D'$ as
$\Ga_K$-modules, then $D$ is (+)de~Rham if and only if $D'$ is
(+)de~Rham.
\end{cor}

We conclude this section by applying the Proposition \ref{ppn-phi-irr}
to $(\vphi,\Ga_K)$-modules $D$ having the property that, as
$\Ga_K$-modules, they are isomorphic to $(t^nB^\dag_{\rig,K})^{\oplus
d}$.  Such $D$ is crystalline, with $D_\crys = (t^{-n}D)^{\Ga_K}$ a
$\vphi$-module over $F$ from which we can recover $D$ completely: $D =
t^nB^\dag_{\rig,K} \otimes_F D_\crys$.

Although we know being crystalline implies being de~Rham, one can also
see that $D$ is de~Rham by way of the proposition: $D_\dif^+$ is
$\Ga_K$-isomorphic to $(t^nK_\infty[\![t]\!])^{\oplus d}$, and
therefore $D_\dif$ is $\Ga_K$-isomorphic to $K_\infty(\!(t)\!)^{\oplus
d}$, which clearly has enough $\Ga_K$-invariants.  Moreover, one finds
that
\begin{align}
\label{eqn-phi-irr1}
\dim_K H^q(\Ga_K,D_\dif^+)
& = d \cdot \dim_K H^q(\Ga_K,t^nK_\infty[\![t]\!]) =
\begin{cases}
d & \text{if }n \leq 0 \\
0 & \text{if }n \geq 1
\end{cases} \\
\intertext{and}
\label{eqn-phi-irr2}
\dim_K H^q(\Ga_K,D_\dif)
& = d \cdot \dim_K H^q(\Ga_K,t^nK_\infty(\!(t)\!)) = d,
\end{align}
for both $q = 0,1$.

\subsection{Cohomology of triangulordinary $(\vphi,\Ga_K)$-modules}
\label{sect-local-proof}

As in \S\ref{sect-definitions}, we set $L = \wh{K^\unr}$.  In this
subsection we prove Theorem \ref{thm-local-main}.

\begin{proof}[Proof of Theorem \ref{thm-local-main}]
We first apply the techniques of Galois descent.

We assume the theorem holds for $K=L$.  Given a triangulordinary $D$,
the theorem over $L$ shows that $D_L$ is (+)de~Rham, and therefore by
Corollary \ref{cor-derham-descent} we know that $D$ is (+)de~Rham.
Moreover, since the $(\Gr_F^n)_L$ are unconditionally crystalline by
the comments concluding \S\ref{sect-local-phi}, we can apply
Proposition \ref{ppn-crys-descent} to deduce that the $\Gr_F^n$
themselves are crystalline, and hence semistable.  Since $D$ is
simultaneously de~Rham and a successive extension of semistable
pieces, \cite[Th\'eor\`eme 6.2]{B1} asserts that $D$ is semistable.
(Note that its proof applies without change to the non-\'etale case.)
Thus we have proved: if $D_L$ is (+)de~Rham, then $D$ is (+)de~Rham
and even semistable, in all cases.

By Corollary \ref{cor-BK-descent}, one has
\[
H^1_{g(+)}(D) = \alpha_L\inv H^1_{g(+)}(D_L).
\]
On the other hand, by its very definition,
\[
H^1_\tord(D;F^*) = \alpha_L\inv H^1_\tord(D_L;F^*_L).
\]
Therefore, if $H^1_\tord(D_L;F^*_L) = H^1_{g(+)}(D_L)$, then
$H^1_\tord(D;F^*) = H^1_{g(+)}(D)$.

The upshot is that we only need to prove (2) and the (+)de~Rham claim
of (1), assuming that
\[
K = L = \wh{K^\unr}, \text{ i.e.\ } k \text{ is algebraically closed.}
\]
Under this assumption, we develop a number of properties of
triangulordinary $(\vphi,\Ga_K)$-modules, organized under the
following lemma.  (In particular, part (3) of the lemma finishes part
(1) of the theorem.)

\begin{lem}\label{lem-tord}
Let $D$ be triangulordinary.  Then the following claims hold.
\begin{enumerate}
\item If $F^1=D$, then $H^q(\Ga,D_\dif^+) = 0$ for $q=0,1$, and
\[
H^1_{g+}(D) = H^1(D) = H^1_\tord(D).
\]
\item One has a decomposition $D_\dif^+ = \bigoplus_n
  (\Gr_F^n)_\dif^+$ as $\Ga_K$-modules.
\item $D$ is de~Rham, and it is +de~Rham if and only if $F^1 = 0$.
\item The natural map $H^1(\Ga_K,D_\dif^+) \to
  H^1(\Ga_K,(D/F^1)_\dif^+)$ is an isomorphism.
\end{enumerate}
\end{lem}

\begin{proof}
(1) To prove the first claim, we proceed by d\'evissage and induction
on the length of $F^*$, and are immediately reduced to the case where
$D = \Gr^n_F$.  But in this situation, Equations
\ref{eqn-phi-irr1}--\ref{eqn-phi-irr2} provide exactly what we desire.
As for the second claim, the first quality follows from the first
claim, while the second equality follows from the fact that $F^1 = D$.

(2) We induct on the length of $F^*$, the case of length $1$ being
trivial.  By twisting, we can assume that $F^0 = D$ and $F^1 \neq D$,
and we must show that the extension class
\[
0 \to (F^1)_\dif^+ \to D_\dif^+ \to (D/F^1)_\dif^+ \to 0
\]
is split.  By Equation \ref{eqn-phi-irr1}, $(D/F^1)_\dif^+ \cong
K_\infty[\![t]\!]^{\oplus d_0}$.  Therefore, as an extension class, we
have
\[
[D_\dif^+] \in \Ext^1_{\Ga_K}({\bf 1}^{\oplus d_0},(F^1)_\dif^+) =
\Ext^1_{\Ga_K}({\bf 1},(F^1)_\dif^+)^{\oplus d_0} =
H^1(\Ga_K,(F^1)_\dif^+)^{\oplus d_0} = 0,
\]
by part (1), since $F^1$ is triangulordinary of weights $\geq 1$.
Hence, the desired extension class is split.

(3) Invoking the decomposition in (2), we have
\[
D_\dif = D_\dif^+[t\inv] = \bigoplus_n (t^nK_\infty[\![t]\!])^{\oplus
d_i}[t\inv] = K_\infty(\!(t)\!)^{\oplus \rank D}.
\]
Therefore, $\dim_K D_\dR = \rank D$, and $D$ is de~Rham.  And,
applying $\Ga_K$-invariants directly to the decomposition in (2),
Equation \ref{eqn-phi-irr1} shows that $D$ is +de~Rham if and only if
$d_n = 0$ for all $n \geq 1$, i.e.\ $F^1 = 0$.

(4) This follows by applying $H^1(\Ga_L,\cdot)$ to the decomposition
of (2), and noting Equation \ref{eqn-phi-irr1}.
\end{proof}

The rest of this section thus aimed at proving claim (2) of the
theorem.  Note that it is precisely at this point that we must work
with the ``g+'' condition and not the ``g'' condition.

Consider the commutative diagram
\[\xymatrix{
H^1(D) \ar[r] \ar[d]_\beta & H^1(\Ga_K,D_\dif^+)
\ar[d]^{\protect\rotatebox[origin=c]{90}{$\sim$}} \\
H^1(D/F^1) \ar[r] & H^1(\Ga_K,(D/F^1)_\dif^+)
},\]
where the isomorphism is by (4) of Lemma \ref{lem-tord}.  The kernel
of the top row is $H^1_{g+}(D)$, so that $H^1_{g+}(D) = \beta\inv
H^1_{g+}(D/F^1)$.  Also, directly from the definition, $H^1_\tord(D) =
\beta\inv H^1_\tord(D/F^1)$.  Therefore, we have reduced to the case
where $F^1 = 0$, i.e.\ $D$ only has weights $\leq 0$.  In this case
(recall $K = \wh{K^\unr}$), by definition $H^1_\tord(D) = 0$, and so
we must show that $H^1_{g+}(D) = 0$ as well, i.e.\ that $H^1(D)
\hookrightarrow H^1(\Ga_K,D_\dif^+)$.

By part (1) of the theorem, assuming that $F^1=0$ means that $D$ is
+de~Rham.  Therefore, $H^1_{g+}(D)$ has an interpretation in terms of
extension classes that are also +de~Rham.  We will harness this
interpretation.

Considering the commutative diagram with exact rows
\[\xymatrix{
& H^1(F^j) \ar[r] \ar[d] & H^1(D) \ar[r] \ar[d] & H^1(D/F^j) \ar[d] \\
0 \ar[r] & H^1(\Ga_K,(F^j)_\dif^+) \ar[r] & H^1(\Ga_K,D_\dif^+) \ar[r]
& H^1(\Ga_K,(D/F^j)_\dif^+) \ar[r] & 0
},\]
an easy diagram chase shows that if the outer vertical arrows are
injective, then so is the middle.  This allows us to induct on the
length of the filtration $F^*$, and reduce to the case where $F^n = D$
and $F^{n+1} = 0$.  In other words, we may assume that, as a
semilinear $\Ga_K$-module, we have $D \cong
(t^nB^\dag_{\rig,K})^{\oplus d}$.  As mentioned in the concluding
comments of \S\ref{sect-local-phi}, such objects are easy to classify:
they are of the form $D = t^nB^\dag_{\rig,K} \otimes_F D_\crys$, where
\[
D_\crys = (D[t\inv])^{\Ga_K} = D^{\Ga_K} = (t^{-n}D)^{\Ga_K}
\]
(remember $n \leq 0$) is a semilinear $\vphi$-module over $F$.

We may further simplify, by reducing the case of general $n$ to the
case when $n=0$ by a descending induction on $n$.  So, we assume the
claim holds for $n+1 \leq 0$, and show it holds for $n$.  We consider
the following commutative diagram with exact rows:
\[\xymatrix@=1.45pc{
H^0(D/tD) \ar[r] \ar[d] & H^1(tD) \ar[r] \ar[d]_a
& H^1(D) \ar[r] \ar[d]_b & H^1(D/tD) \ar[d] \\
H^0(\Ga_K,D_\dif^+/tD_\dif^+) \ar[r] & H^1(\Ga_K,tD_\dif^+) \ar[r]
& H^1(\Ga_K,D_\dif^+) \ar[r] & H^1(\Ga_K,D_\dif^+/tD_\dif^+)
}.\]
In the bottom row, the first and last groups are respectively
isomorphic to
\[
H^q(\Ga_K,K_\infty t^n)^{\oplus d} \text{ for } q=0,1,
\]
and thus classically seen to be trivial; hence, the bottom middle
arrow is an isomorphism.  Using \cite[Lemma 3.2(1--2)]{L}, we have
\begin{multline}\label{eqn-local-KL}
D/tD \cong (t^nB^\dag_{\rig,K}/t^{n+1}B^\dag_{\rig,K})^{\oplus d} \\
= \rlim_r (t^nB^{\dag,r}_{\rig,K}/t^{n+1}B^{\dag,r}_{\rig,K})^{\oplus d}
= \rlim_r \prod_{m \geq n(r)} (K'_m t^n)^{\oplus d}.
\end{multline}
Examining the Herr complex
\[
(D/tD)^{\Delta_K} \xrightarrow{(\vphi-1,\ga-1)} (D/tD)^{\Delta_K}
\oplus (D/tD)^{\Delta_K} \xrightarrow{(1-\ga) \oplus (\vphi-1)}
(D/tD)^{\Delta_K},
\]
we immediately deduce from Equation \ref{eqn-local-KL} that
$H^0(D/tD)$ vanishes, because $n \neq 0$.  On the other hand, one
easily uses Equation \ref{eqn-local-KL} and \cite[Lemma 3.2(3--4)]{L}
to calculate that $H^1(D/tD)$ is isomorphic to $\rlim_m
\left[((K'_m)^{\Delta_K}t^n)/(\ga-1)\right]^{\oplus d}$, and each term
of this limit is zero, since $n \neq 0$.  Hence, the top middle arrow
in our commutative diagram is an isomorphism too.  Notice that $n-1$
is not a $\vphi$-slope on $D$ if and only if $n$ is not a
$\vphi$-slope on $tD$ (which is triangulordinary of all weights
$n+1$).  Therefore, the inductive hypothesis applies to $tD$, and $b$
is injective; we conclude that $a$ is also injective.  Thus it
suffices to treat the case where $n=0$.

Recall that we want to show that $H^1_{g+}(D) = 0$.  In other words,
given a class $c \in H^1(D)$, represented by an extension $E_c$, we
want to show that if $E_c$ is +de~Rham then it must be split.  If it
is +de~Rham then, invoking \cite[Th\'eor\`eme 6.2]{B1} again, it must
be semistable.  We will show that every semistable extension $E_c$ is
crystalline, {\it under our hypothesis on the $\vphi$-slopes of $D$}.
Then, we will show that every crystalline extension is split.

So let $E = E_c$ be given, and assumed to be semistable.  We write
$E_\st = \bfD_\st(E)$ for the associated filtered $(\vphi,N)$-module;
showing that $E$ is crystalline is tantamount to showing that $N=0$ on
$E_\st$.  In fact, since $D$ is crystalline, one has $N = 0$ on the
corank $1$ subspace $D_\crys \subset E_\st$.  To treat the remainder
of $E_\st$, consider the $\vphi$-modules over $F$ that underlie the
the exact sequence
\[
0 \to D_\crys \to E_\st \to {\bf 1}_\crys \to 0.
\]
By the Dieudonn\'e--Manin theorem, the category of $\vphi$-modules
over $F$ is semisimple (recall that $k$ is algebraically closed), so
we may split this extension of $\vphi$-modules, i.e.\ choose a
$\vphi$-fixed vector $e \in E_\st$ that spans the complement of
$D_\crys$.  We will know that $N=0$ on $E_\st$ as soon as we know that
$N(e)=0$.  To see this, remember that $D$ is $\Ga_K$-isomorphic to
$(B^\dag_{\rig,K})^{\oplus d}$, and so the $\vphi$-slopes (\`a la
Dieudonn\'e--Manin) on $D_\crys$ are equal to the $\vphi$-slopes (\'a
la Kedlaya) on $D$, which by hypothesis are not equal to $-1$, while
the $\vphi$-slope of $e$ is $0$.  Denoting $E_\st^{(\la)}$ for the
slope-$\la$ part of $E_\st$, recall that $N\vphi = p\vphi N$, and
hence $N(E_\st^{(\la)}) \subseteq E_\st^{(\la-1)}$.  On the other
hand, we have
\[
N(e) \in N(E_\st^{(0)}) \subseteq E_\st^{(-1)} = 0.
\]
Therefore, $E$ is crystalline.

Given an extension $E$ of $D$ that is crystalline, we show that it is
trivial.  In fact, applying Dieudonn\'e--Manin just as above, we find
that $E_\st = E_\crys$ is split as a $\vphi$-module:
\[
E_\crys = D_\crys \oplus {\bf 1}_\crys.
\]
But we may recover $E$ {\it as a $(\vphi,\Ga_K)$-module} from
$E_\crys$.  In fact,
\begin{multline*}
E = B^\dag_{\rig,K} \otimes_F E_\crys = B^\dag_{\rig,K} \otimes_F
 (D_\crys \oplus {\bf 1}_\crys) \\ = (B^\dag_{\rig,K} \otimes_F
 D_\crys) \oplus (B^\dag_{\rig,K} \otimes_F {\bf 1}_\crys) = D \oplus
 {\bf 1},
\end{multline*}
with $\vphi$ acting diagonally and $\Ga_K$ acting on the left
$\otimes$-factors.  This shows that $E = E_c$ is split as an extension
of $(\vphi,\Ga_K)$-modules, and its corresponding class $c \in H^1(D)$
is trivial.  This completes the proof.
\end{proof}

\begin{rem}
A curious byproduct of the final step of the argument is that, for $D$
in the special form treated there, if an extension $E$ of $D$ as a
$\Ga_K$-module admits any $\vphi$-structure, then it admits only one
$\vphi$-structure.
\end{rem}

\begin{rem}
The hypothesis on the $\vphi$-slopes of $D$ really is necessary when
working with general $(\vphi,\Ga_K)$-modules, as the following example
shows.  The trouble seems to be that when we have left the category of
Galois representations, i.e.\ {\it \'etale} $(\vphi,\Ga_K)$-modules
and {\it admissible} filtered $(\vphi,N,G_K)$-modules, there are
simply too many objects to be handled via ordinary-theoretic
techniques.  Cf.\ the failure of the above proof to apply to the ``g''
local condition.
\end{rem}

\begin{exmp}\label{exmp-bad}
Consider the filtered $(\vphi,N)$-module $E_\st =
\text{span}(e_0,e_{-1})$, with $\Fil^0 = E_\st$, $\Fil^1 = 0$, and
$\vphi$ and $N$ given by
\[
\vphi(e_0)=e_0,\ \vphi(e_{-1})=p\inv e_{-1},
\qquad \text{and} \qquad
N(e_0)=e_{-1},\ N(e_{-1})=0.
\]
Then $E_\st$ corresponds to a $(\vphi,\Ga_K)$-module $E$ of rank $2$
which is not \'etale, is semistable of Hodge--Tate weights both $0$,
and is {\it not crystalline}.  Such an object is unheard of in the
classical setting.

Notice that $D_\st = \text{span}(e_{-1})$ corresponds to a subobject
$D$ of $E$, with quotient object isomorphic to ${\bf 1}_\crys$ (taking
$e_0$ to the standard basis element).  Thus, $E$ represents a
nontrivial extension class in $H^1(D)$, which is actually +de~Rham
because its Hodge--Tate weights are all $0$.  On the other hand, $D$
is only triangulordinary with respect to $F^0=D$, $F^1=0$, and so if
(for example) $K=\wh{K^\unr}$ then $H^1_\tord(D) = 0$, and
$H^1_\tord(D) \subsetneq H^1_{g+}(D)$.
\end{exmp}

The manner of the reduction steps in the proof of the theorem show
that all counterexamples to its conclusion, outside the context of the
slope hypothesis, arise by some manipulation (twisting, extensions,
descent) from the above example.

We will see in \S\ref{sect-local-ex} that the above counterexample
really does occur as a graded piece within \'etale
$(\vphi,\Ga_K)$-modules arising in nature, namely in the setting of a
modular form with good reduction at $p$ and slopes $1,k-2$ (where $k$
is the weight of $f$).  It is unclear to the author whether its
obstruction can be worked around, even in explicit examples.

\subsection{Comparison with Bloch--Kato, ordinary and trianguline}
\label{sect-compare}

Since the reader might be wondering what ``triangulordinary'' means,
we explain how triangulordinary representations and Selmer groups
relate to common notions due to Bloch--Kato, Greenberg, and Colmez.

All the examples that motivate this work take place when $D$ is
\'etale, so that $D = \bfD^\dag_\rig(V)$ for a {\it bona fide}
$p$-adic representation $V$ of $G_K$.  In order to place Theorem
\ref{thm-local-main} into context, we recall the following facts,
which are essentially due to Bloch--Kato \cite{BK}.

\begin{ppn}[\cite{BK}]\label{ppn-local-BK}  Let $V$ be de~Rham.  Then
the following claims hold.
\begin{enumerate}
\item One always has $H^1_g(V) = H^1_{g+}(V)$, these two items being
  defined in \S\ref{sect-galois-descent}.
\item If $V$ is semistable and $\bfD_\crys(V)^{\vphi=p\inv} = 0$, then
  $H^1_g(V) = H^1_f(V)$.
\end{enumerate}
In the second item, the local condition $H^1_f(V)$ consists of those
extension classes that are split after tensoring with $B_\crys$, and
provides the correct Selmer group with which to state Bloch--Kato's
conjectural analytic class number formula.
\end{ppn}

Thus, in the triangulordinary setting, the local condition
$H^1_\tord(V;F)$ usually measures $H^1_f(V)$, and hence the
triangulordinary Selmer group computes the Bloch--Kato Selmer group,
which is of intrinsic interest.

We find it helpful to have access to the following equivalent
formulation.

\begin{altdefn}\label{alt-defn}
We remind the reader that by Theorem \ref{thm-local-main}(1), every
triangulordinary $(\vphi,\Ga_K)$-module is semistable.  Thus, there is
no loss of generality in assuming this is the case from the outset.

Given a semistable $(\vphi,\Ga_K)$-module $D$, the discussion at the
conclusion of \S\ref{sect-phigamma-monodromy} relates semistable
subobjects of $D$ to subobjects of $D_\st$.  We find that the
triangulordinary filtrations $F^* \subseteq D$ are in a natural
correspondence with filtrations $F^* \subseteq D_\st$ by
$(\vphi,N)$-stable subvector spaces (each equipped with its Hodge
filtration induced by $D_\st$), such that each $\Gr_F^n D_\st$ has all
its induced Hodge--Tate weights equal to $n$ (i.e., its induced Hodge
filtration is concentrated in degree $-n$).

Given a corresponding pair $F^* \subseteq D$ and $F^* \subseteq
D_\st$, the gradeds $\Gr_F^n D$ and $\Gr_F^n D_\st$ are linked by the
formula
\[
\Gr_F^n D \cong t^nB^\dag_{\rig,K} \otimes_F \Gr_F^n D_\st
\]
as $(\vphi,\Ga_K)$-modules.  Therefore, the $\vphi$-slopes on $\Gr_F^n
D$ are $n + {}$the $\vphi$-slopes on $\Gr_F^n D_\st$.  As concerns
Theorem \ref{thm-local-main}(2), the requirement that for all $n \leq
0$, $n-1$ not be a $\vphi$-slope on $\Gr_F^n D$ becomes the condition
that, for all such $n$, the graded $\Gr_F^n D_\st$ does not contain
the $\vphi$-slope $-1$.
\end{altdefn}

\begin{exmp}[Relation to ordinary representations]  Let us
see how ordinary representations, defined by Greenberg in \cite{G1},
fit into our context.  We are given a Galois representation $V$, so
that $D = \bfD^\dag_\rig(V)$ is \'etale.  The ordinary hypothesis is
that $V$ admits a decreasing filtration $F^* \subseteq V$ by
$G_K$-stable subspaces, such that, for each $n$, the representation
$\chi_\cycl^{-n} \otimes \Gr_F^n V$ is unramified.  Applying
$\bfD^\dag_\rig$, we obtain a decreasing, $(\vphi,\Ga_K)$-stable
filtration $F^* \subseteq D$ by $B^\dag_{\rig,K}$-direct summands,
such that each $\Gr_F^n D_L$ is \'etale and $\Ga_L$-isomorphic to
$(t^nB^\dag_{\rig,L})^{\oplus d_n}$.  Conversely, given such a
filtration on $D$, Theorem \ref{thm-phigamma-equiv} produces an
ordinary filtration on $V$.  Thus, given an \'etale $D$, the ordinary
hypothesis is a strengthening of the triangulordinary hypothesis to
require all the graded pieces to be \'etale.

In the language of filtered $(\vphi,N)$-modules, a triangulordinary
filtration $F^* \subseteq D_\st$ corresponds to an ordinary filtration
precisely when all the $\Gr_F^n D_\st$ are admissibly filtered, which
means here that each $\Gr_F^n D_\st$ is of pure $\vphi$-slope $-n$.

Moreover, Theorem \ref{thm-local-main}(2) always applies to ordinary
representations.  Namely, for all $n \leq 0$, the number $n-1$ (which
is $\leq -1$) never occurs as a $\vphi$-slope on $\Gr_F^n D$ because
the latter is \'etale.  Thus, this theorem provides a generalization
of Flach's result \cite[Lemma 2]{Fl} from the case where $K=\bbQ_p$ to
the case of arbitrary perfect residue field.

We alert the reader to the fact that, although $V$ admits at most one
ordinary filtration, it may admit many different triangulordinary
filtrations, as we will see below.
\end{exmp}

Before discussing trianguline representations, we point out that our
entire theory works perfectly well with the $E$-coefficients replacing
the $\bbQ_p$-coefficients of Galois representations, for any finite
extension $E/\bbQ_p$.

\begin{exmp}[Relation to trianguline representations]  We now
determine when a triangulordinary $(\vphi,\Ga_K)$-module is
trianguline, and give some comments about the converse.  (More
precisely, we discuss when one can modify a triangulordinary
filtration into a trianguline one, and vice versa.)  Recall that,
following Colmez \cite[\S0.3]{C}, a $(\vphi,\Ga_K)$-module $D$ is {\it
trianguline} if it is a successive extension of rank $1$ objects,
i.e.\ if there exists a decreasing, separated and exhaustive
filtration $F^* \subseteq D$ by $(\vphi,\Ga_K)$-stable
$B^\dag_{\rig,K}$-direct summands, with each graded of rank $1$.  We
call the latter a {\it trianguline filtration}.  When $D$ is
semistable, these correspond precisely to {\it refinements} in the
sense of Mazur: complete flags in $D_\st$ by $(\vphi,N)$-stable
subspaces.

A triangulordinary $D$ is trianguline precisely when the gradeds
$\Gr_F^*$ of the triangulordinary filtration $F^*$ are themselves
trianguline.  Sufficiency is clear.  For necessity, note that $D$ is
semistable, so assume given $D_\st$, and think of triangulordinary
filtrations as being $(\vphi,N)$-stable ones on $D_\st$.  Given our
triangulordinary filtration and any (other) trianguline filtration on
D, taking the intersections of the two filtrations gives a refinement
of the triangulordinary filtration with rank one gradeds.  (One could
avoid using filtered $(\vphi,N)$-modules by converting this last step
into the language of B\'ezout domains.)

In any case, with $D$ triangulordinary, each the $\Gr_F^*$ is
crystalline (by Corollary \ref{cor-crys-HT-n}), so $D$ is trianguline
if and only if the $(\Gr_F^*)_\crys$ admit refinements.  Clearly, this
is the case precisely when $D_\crys$ is an extension of
one-dimensional $\vphi$-stable subspaces.  If the residue field $k$ is
finite, then one can always replace the coefficient field $E$ by a
finite extension in order to achieve this.

As for the converse, since triangulordinary $D$ are always semistable,
we ask when a semistable trianguline $D$ is triangulordinary.  It
turns out that {\it not all} such $D$ are triangulordinary.  For a
semistable trianguline $D$ that is not triangulordinary, we consider
$E$ as in Example \ref{exmp-bad}.  Being constructed out of $E_\st$,
it is semistable; being an extension of $D$ by ${\bf 1}$ it is
trianguline.  Both its Hodge--Tate weights are $0$, so a putative
triangulordinary filtration $F^* \subseteq E$ would have $E =
\Gr_F^0$; by Corollary \ref{cor-crys-HT-n}, in order for $E$ to be
triangulordinary it must be crystalline.  But $E$ is not crystalline.

The above example shows that, roughly, having a nonzero monodromy
operator acting within fixed a Hodge--Tate weight part is an
obstruction to being triangulordinary.  Let us assume that this is not
the case for $D$, and suppose we are given a trianguline filtration
$F^* \subseteq D$.  In order for $D$ to be triangulordinary, we must
be able to arrange that the Hodge--Tate weights of the $\Gr_F^*$ are
nondecreasing, because then, weakening $F^*$ so that each $\Gr_F^n$
has all Hodge--Tate weights equal to $n$, the resulting gradeds must
also be crystalline (by our rough assumption), hence
$\Ga_L$-isomorphic to $(t^nB^\dag_{\rig,L})^{\oplus d_n}$ by Corollary
\ref{cor-crys-HT-n}.  In order to rearrange $F^*$ to have Hodge--Tate
weights in nondecreasing order, we must be able to break up any
extension between adjacent gradeds that are in the wrong order.  Given
the intermediate extension of filtered $(\vphi,N)$-modules
\[
0 \to (\Gr_F^{n+1})_\crys \to (F^n/F^{n+2})_\st \to (\Gr_F^n)_\crys
\to 0
\]
whose Hodge--Tate weights are in decreasing order, one easily checks
that any $(\vphi,N)$-equivariant splitting will do.  But
$(\vphi,N)$-equivariant splittings, in turn, might not exist.  The
$\vphi$-structure itself could be nonsemisimple; by
Dieudonn\'e--Manin, this would require the crystalline $\vphi$-slopes
on the adjacent gradeds to be equal, and the $\vphi$-extension would
only necessarily split upon restriction to $\wh{K^\unr}$.  Assuming
otherwise, that one can find a $\vphi$-eigenvector $v \in
(F^n/F^{n+2})_\st$ mapping onto a basis for $(\Gr_F^n)_\crys$, the
extension is split as $(\vphi,N)$-module if and only if $N(v) = 0$,
which might or might not hold.

In summary, the trianguline condition is roughly more general than the
triangulordinary condition.  Triangulordinary representations are
semistable, and are trianguline when their filtrations may be further
subdivided to have gradeds of rank $1$; the latter always happens
after an extension of coefficients when $k$ is finite.

Trianguline $D$ may be highly nonsemistable due to continuous
variation of Sen weights.  When they are semistable, they may fail to
be triangulordinary, if they have nontrivial extensions with the wrong
ordering of Hodge--Tate weights, or if they have extensions of common
Hodge--Tate weight that are semistable but not crystalline.
\end{exmp}

\subsection{Examples of triangulordinary representations}
\label{sect-local-ex}

In this section we explain when abelian varieties and modular forms
are triangulordinary.  In passing, we gather for easy reference
descriptions of the invariants of the cyclotomic character and modular
forms.  Since many different normalizations are used in the
literature, we have made an effort to organize them systematically.

Let us begin with some discussion of normalizations.  The general
rules are summarized in the following table.  The initial column says
what kind of motive we are dealing with: one cut out of homology or
cohomology.  The first property is which Frobenius operator on
$\ell$-adic realizations has $\ell$-adic {\it integer} eigenvalues.
Next is which power of crystalline Frobenius, $\vphi$ or $\vphi\inv$,
has $p$-adic integer eigenvalues.  Then come the degrees in which we
expect to see jumps in the Hodge filtration.  Finally, we see which
powers of the cyclotomic character tend to appear in the action of
$\Ga_\bbQp$ on basis elements of the $\bfD_\dif^+$ (which is a rough
indication of the $\Ga_\bbQp$-action on $\bfD^\dag_\rig$); these are
the jumps we expect to see in a triangulordinary filtration.

\vskip 12pt

\begin{tabular}{|l|l|l|}
\hline
& homological & cohomological \\
\hline
\hline
$\ell$-adic $\Frob$ & arithmetic & geometric \\
\hline
\hline
crystalline $\vphi$ & $\vphi\inv$ & $\vphi$ \\
\hline
Hodge jumps & nonpositive & nonnegative \\
\hline
\hline
$\Ga_\bbQp$ on $\bfD_\dif^+$ & nonnegative & nonpositive \\
\hline
$\nabla$ord jumps & nonnegative & nonpositive \\
\hline
\end{tabular}

\vskip 12pt

We give three reminders: the $\vphi$ on $\bfD^\dag_\rig$ is always
\'etale, the cyclotomic character and Tate modules of abelian varieties
are {\it homological} objects, and this table is invariant under the
choice of sign of the Hodge--Tate weight of the cyclotomic character.
(In this text, the cyclotomic character has Hodge--Tate weight $+1$.)

\begin{exmp}[The cyclotomic character]  We consider the $p$-adic
cyclotomic character $\chi_\cycl$ as a $1$-dimensional $\bbQ_p$-vector
space $\bbQ_p \cdot e_{\chi_\cycl}$, equipped with a $\bbQ_p$-linear
$G_K$-action via $g(e_{\chi_\cycl}) = \chi_\cycl(g)e_{\chi_\cycl}$.
One has $\bfD^\dag_\rig(\chi_\cycl^n) = B^\dag_{\rig,K} \cdot (1
\otimes e_{\chi_\cycl}^{\otimes n})$ and $\bfD_\crys(\chi_\cycl) = F
\cdot (t^{-n} \otimes e_{\chi_\cycl})$.  From these, we derive the
following table, giving actions on the basis vectors just mentioned.

\vskip 12pt

\begin{tabular}{|l||l|l||l|l|}
\hline
$\ell$-adic $\Frob_\ell^\text{arith}$
  & $\vphi$ on $\bfD_\crys$
  & $\Gr^? \neq 0$
  & $\vphi$ on $\bfD^\dag_\rig$
  & $\Ga_\bbQp$ on $\bfD\dag_\rig$ \\
\hline
$p^n$ & $p^{-n}$ & $-n$ & $1$ & $\chi_\cycl^n$ \\
\hline
\end{tabular}

\vskip 12pt

Finally, we point out that the powers of the cyclotomic character are
all ordinary, hence triangulordinary.  Since they are one-dimensional,
the ordinary filtration is the only choice of triangulordinary
filtration.
\end{exmp}

\begin{exmp}[Abelian varieties]
Take a semistable abelian variety $B$ over $K$ of dimension $d \geq
1$, and consider $D_\st = \bfD_\st(V)$, with $V = T_pB \otimes \bbQ$
the $p$-adic Tate module up to isogeny.  Thus, when dealing with
abelian varieties, we are primarily concerned with homology.  It is
well-known that the Hodge--Tate weights of $B$ are $0$ and $1$, each
with multiplicity $d$, and this tells us that the Hodge filtration
$H^* \subseteq D_\dR$ satisfies $\dim_K \Gr_H^0 = \dim_K \Gr_H^{-1} =
d$, and our triangulordinary filtration $F^* \subseteq D_\st$ must
satisfy $\rank \Gr_F^0 = \rank \Gr_F^1 = d$, and all other gradeds are
trivial.  So, $F^*$ consists of the single datum of a
$(\vphi,N)$-stable $F$-subspace $F^1 \subset D_\st$ of dimension $d$.

Weak admissibility, here, means: nonzero slopes do not meet the Hodge
filtration $H$.  Ordinary means that half these slopes are $0$ (can
lie anywhere), and half are $-1$ (cannot lie in $H$: span a weakly
admissible submodule).  Triangulordinary means, one can find half of
these slopes, $\vphi$-stably, not contained in $H$.  This means that
corresponding subspace $F^1$ has induced Hodge filtration concentrated
in degree $-1$, which automatically forces $\Gr_F^0 = D_\st/F^1$ to
have induced Hodge filtration concentrated in degree $0$.

Thus, in short, a triangulordinary filtration for $V$ consists of a
$d$-dimensional $(\vphi,N)$-stable subspace $F^1 \subseteq D_\st$ such
that $F^1 \otimes_F K$ is complementary to the Hodge filtration $H^0
\subset D_\dR$.

We stress that, because we are in a homological situation, the
$\vphi$-slopes on $\bfD_\st(V)$ are nonpositive.  In order for Theorem
\ref{thm-local-main}(2) to apply, all we need is that $-1$ does not
occur as a $\vphi$-slope on the quotient $\Gr_F^0 D_\crys =
D_\crys/F^1$, or, equivalently, that every instance of slope $-1$
occurs within $F^1$.  (This hypothesis is a variant of a ``noncritical
slope'' condition.)

Let us illustrate the above with some examples, assuming, for
simplicity, that $B$ has good reduction and our coefficients are
$E=\bbQ_p$:

Suppose $B$ has slopes $-1,-2/3,-1/3,0$.  Then $B$ is nonordinary, and
always triangulordinary: for the triangulordinary filtration, one can
take either any of the two spaces with slopes $(-1,-2/3)$, $(-1,1/3)$.
When the $0$-slope is not in $H$, one gets two more options.  The
theorem applies to the first two of these, but not to the possible
latter two.

If $B$ has slopes $-1,-1/2,-1/2,0$, then it is nonordinary, and always
triangulordinary: its filtrations include the two slope $(-1,-1/2)$
spaces, and, if the slope $0$ space is not in $H$, then the two slope
$(-1/2,0)$ spaces are also valid.  (In particular, having slopes equal
to $-1/2$ is {\it not} necessarily an obstruction.)  The theorem
applies to the first filtrations, and not to the second ones.

Let $B$ have slopes $-1,-1,-1,-1/2,-1/2,0,0,0$, and assume $\vphi$
acts irreducibly on its pure-slope spaces.  Then $B$ is not
triangulordinary, simply because there are no $\vphi$-stable subspaces
with half the total dimension.

We leave it to the reader to examine, when the residue field $k$ is
finite, what additional possibilities occur after enlarging the
coefficient field $E$ to break the pure-slope spaces into extensions
of one-dimensional $\vphi$-stable spaces.
\end{exmp}

In the following example, fix a coefficient field $E$.

\begin{exmp}[Elliptic modular eigenforms]\label{exmp-MFs}
Let $f \in S_k(\Ga_1(M),\psi;E)$ be a normalized elliptic modular
cuspidal eigenform such that $k \geq 2$, having $q$-expansion $\sum
a_nq^n$.  Deligne has associated to $f$ a $2$-dimensional $E$-valued
representation of the absolute Galois group of $\bbQ$, which is
unramified away from $Mp\infty$ and de~Rham at $p$.  It is absolutely
irreducible, so it is characterized (up to a scalar multiple) by its
characteristic polynomials; by Chebotarev, it is enough to know the
polynomials of the Frobenius elements $\Frob_\ell$ at primes $\ell
\nmid Mp$.  For such $\ell$, one has
\[
\text{trace}(\Frob_\ell) = a_\ell
\quad\text{ and }\quad
\det(\Frob_\ell) = \ell^{k-1}\psi(\ell).
\]
One vagueness that typically makes these matters confusing is whether
$\Frob_\ell$ refers to the arithmetic or the geometric Frobenius.  In
the case where the above equations describe the arithmetic Frobenius,
we say that the representation is the {\it homological} normalization,
and denote it $V_f^\text{hom}$.  When they apply to the geometric
Frobenius, we say that the representation is the {\it cohomological}
variant, and we denote it $V_f^\text{coh}$.  These names originate in
whether $V_f^?$ is found within the \'etale homology or cohomology of
Kuga--Sato varieties, respectively.  In either case, we only consider
the restriction of the $G_\bbQ$-action to a decomposition group
$G_\bbQp$.

In what follows, we make the following hypothesis: $f$ is semistable
at $p$, and the operator $\vphi$ on $\bfD_\st(V_f^\text{coh})$ has
{\it distinct} eigenvalues $\la,\mu$ lying in $E$.  (This is
equivalent to requiring the same on $\bfD_\st(V_f^\text{hom})$, with
the the roots $\la\inv,\mu\inv \in E$.)  We order the roots so that
$\ord_p \la \leq \ord_p \mu$; one has $\ord_p \la + \ord_p \mu = k-1$,
with $\ord_p \la = 0$ if and only if $f$ is ordinary at $p$.

The cohomological normalization has
\[
\bfD_\st(V_f^\text{coh}) = E \cdot e_\la \oplus E \cdot e_\mu
\quad \text{with} \quad
\left\{\begin{array}{ll}
  \vphi(e_\nu) = \nu e_\nu, \\
  \bfD = F^0 \supsetneq F^1 = \cdots = F^{k-1} \supsetneq F^k = 0.
\end{array}\right.
\]
The ``weak admissibility'' condition means that $e_\la \notin F^1$,
and $e_\mu \notin F^1$ unless possibly if $\ord_p \mu = k-1$ (in which
case $f$ is split ordinary at $p$).  The monodromy operator $N$ is
only nonzero when $p \mid M$; if $N \neq 0$ then $\ord_p \mu = \ord_p
\la + 1$, and $N$ is determined by $N(e_\mu) = e_\la$ and $N(e_\la) =
0$ (up to rescaling $e_\la$).

Since the nonzero gradings for the Hodge filtration are $0,k-1$, each
with one-dimensional graded, the nonzero triangulordinary gradings are
$1-k,0$, each with one-dimensional graded.  Thus we consider the
rank-one $\vphi$-stable subspaces of $\bfD_\st(V_f^\text{coh})$, and
their corresponding $(\vphi,\Ga_\bbQp)$-modules.  Let $\nu$ be one of
$\la$ or $\mu$ and let $\nu'$ be the other; let $D_\nu = E \cdot e_\nu
\subseteq \bfD_\st$ and $D'_\nu = \bfD_\st/D_\nu$.  Then, for
comparison to the cyclotomic character, one has the following table.
The parenthetical values are used precisely when $f$ is split ordinary
at $p$ and $\nu = \mu$.

\vskip 12pt

\begin{tabular}{|l|l||l|l|}
\hline
$\vphi$ on $D_\nu$
  & $\Gr^? \neq 0$
  & $\vphi$ on $(D_\nu)^\dag_\rig$
  & $\Ga_\bbQp$ on $(D_\nu)^\dag_\rig$ \\
\hline
$\nu$ & $0$ ($k-1$) & $\nu$ (${\nu'}\inv$) & $1$ ($\chi_\cycl^{1-k}$) \\
\hline
\hline
$\vphi$ on $D'_\nu$
  & $\Gr^? \neq 0$
  & $\vphi$ on $(D'_\nu)^\dag_\rig$
  & $\Ga_\bbQp$ on $(D'_\nu)^\dag_\rig$ \\
\hline
$\nu'$ & $k-1$ ($0$) & $\nu\inv$ ($\nu'$) & $\chi_\cycl^{1-k}$ ($1$)\\
\hline
\end{tabular}

\vskip 12pt

\noindent The triangulordinary hypothesis on $F^*$ requires that $F^0$
not meet the Hodge filtration $H^1 = H^{k-1} \subset D_\st \otimes_F
K$, and, if this holds, then one obtains for free that $\Gr^{1-k}
D_\st = D_\st/F^0$ has induced Hodge--Tate weight $k-1$, as is
required.  Examining the above table, we see that $e_\nu$ always
defines a trianguline filtration, and that $e_\nu$ defines a
triangulordinary filtration except in the parenthetical (split
ordinary) case.  In the split ordinary case, taking $\nu = \la$ still
gives a triangulordinary filtration.  Also, theorem
\ref{thm-local-main}(2) always applies, because the only nonzero
$\Gr^n$ with $n \leq 0$ is with $n=0$, and the only $\vphi$-slope
occurring there is nonnegative.

We obtain descriptions of the homological normalization by taking
$E$-linear duals of everything above.  Namely, $\bfD_\st$ has
\[
\bfD_\st(V_f^\text{hom})
  = E \cdot e_{\la\inv} \oplus E \cdot e_{\mu\inv}
\text{ with }
\left\{\begin{array}{ll}
  \vphi(e_{\nu\inv}) = \nu\inv e_{\nu\inv}, \\
  \bfD = F^{1-k} \supsetneq F^{2-k} = \cdots = F^0 \supsetneq F^1 = 0.
\end{array}\right.
\]
The ``weak admissibility'' condition means that $e_{\mu\inv} \notin
F^1$, and $e_{\la\inv} \notin F^1$ unless possibly if $\ord_p \la = 0$
(in which case $f$ is split ordinary at $p$).  The monodromy operator
$N$ is only nonzero when $p \mid M$; if $N \neq 0$ then $\ord_p \mu =
\ord_p \la + 1$, and $N$ is determined by $N(e_{\la\inv}) =
e_{\mu\inv}$ and $N(e_{\mu\inv}) = 0$ (after perhaps rescaling
$e_{\mu\inv}$).  Note the role reversal between $\mu$ and $\la$; this
is only because $\ord_p \mu\inv \leq \ord_p \la\inv$.

Our triangulordinary filtration must have one-dimensional nonzero
gradeds in degrees $0,1$, and so we consider the rank-one
$\vphi$-stable subspaces of $\bfD_\st(V_f^\text{hom})$, and their
corresponding $(\vphi,\Ga_\bbQp)$-modules.  Let $\nu$ be one of $\la$
or $\mu$ and let $\nu'$ be the other; let $D_{\nu\inv} = E \cdot
e_{\nu\inv} \subseteq \bfD_\st$ and $D'_{\nu\inv} =
\bfD_\st/D_{\nu\inv}$.  Again, we have the following table.  The
parenthetical values are used precisely when $f$ is split ordinary at
$p$ and $\nu\inv = \la\inv$.

\vskip 12pt

\begin{tabular}{|l|l||l|l|}
\hline
$\vphi$ on $D_{\nu\inv}$
  & $\Gr^? \neq 0$
  & $\vphi$ on $(D_{\nu\inv})^\dag_\rig$
  & $\Ga_\bbQp$ on $(D_{\nu\inv})\dag_\rig$ \\
\hline
$\nu\inv$ & $1-k$ ($0$) & $\nu'$ ($\nu\inv$) & $\chi_\cycl^{k-1}$ ($1$) \\
\hline
\hline
$\vphi$ on $D'_{\nu\inv}$
  & $\Gr^? \neq 0$
  & $\vphi$ on $(D'_{\nu\inv})^\dag_\rig$
  & $\Ga_\bbQp$ on $(D'_{\nu\inv})^\dag_\rig$ \\
\hline
${\nu'}\inv$ & $0$ ($1-k$) & ${\nu'}\inv$ ($\nu$) & $1$ ($\chi_\cycl^{k-1}$)
\\
\hline
\end{tabular}

\vskip 12pt

\noindent In particular, we see that $e_{\nu\inv}$ always defines a
trianguline filtration, and that $e_{\nu\inv}$ defines a
triangulordinary filtration except in the parenthetical (split
ordinary) case.  In the split ordinary case, taking $\nu = \mu\inv$
still gives a triangulordinary filtration.  As concerns Theorem
\ref{thm-local-main}, $\Gr^{1-k}$ always has nonnegative slope, and
hence presents no obstruction.  But $\Gr^0$ has slope $\ord_p \nu +
1-k$, which is equal to $-1$ when $\ord_p \nu = k-2$.  Thus, the
theorem does not apply to modular forms with triangulordinary
filtration determined by a $\vphi$-eigenvalue of slope $k-2$.

We invite the reader to check that the above conclusions for
$V_f^\text{hom}$ agree with the conclusions made, in the case $k=2$,
for $T_pB_f \otimes \bbQ$, where $B_f$ is the corresponding modular
abelian variety.
\end{exmp}

\begin{rem}
The example above should generalize readily to Hilbert modular forms.
\end{rem}

\begin{rem}
It is {\it extremely unusual} to naturally encounter a theorem that
applies only to modular forms with a $U_p$-slope $\neq k-2$, as does
Theorem \ref{thm-local-main} for $V_f^\text{hom}$.  The only special
slopes are usually $0$, $k-1$, and $\frac{k-1}{2}$.
\end{rem}

\section{Variational program}\label{sect-global}

In this section we describe a conjectural program for obtaining
triangulordinary filtrations, and hence Selmer groups, for families of
Galois representations.  Our primary guide here is Greenberg's
variational viewpoint, described in \cite{G2}.

As a main example, we consider the eigencurve of Coleman--Mazur.  We
show how our program would recover results of Kisin (see \cite{Ki}),
and interpolate his Selmer groups for overconvergent modular forms of
finite slope into a Selmer module over the entire eigencurve.  Note
that we use the {\it homological} normalization to maintain
consistency with Greenberg on the ordinary locus, and with the
statements of Kisin.

We retain the notations and conventions of the preceding sections,
with the additional assumption that our local fields $K$ have {\it
finite} residue fields (since this is required by Berger--Colmez in
\cite{BC}).  A careful reading of Berger--Colmez might allow this
restriction to be removed.

\subsection{Interpolation of $(\vphi,\Gamma_K)$-modules}

We construct a families of $(\vphi,\Ga_K)$-modules over rigid analytic
spaces corresponding to families of $p$-adic representations of $G_K$,
using the theory of Berger--Colmez.  In their work, the base of the
family is a $p$-adic Banach space $S$.  By an $S$-representation of
$G_K$, we mean a finite free $S$-module $V$ equipped with a
continuous, $S$-linear $G_K$-action.  In order to get a $p$-adic Hodge
theory for $V$, we must assume the mild condition that $S$ is a {\it
coefficient algebra} as in \cite[\S2.1]{BC}

We require some terminology from $p$-adic functional analysis.  Given
a $p$-adic Banach algebra $S$ with norm $|\cdot|_S$ and a Fr\'echet
space $T$ with norms $\{|\cdot|_i\}_{i \in I}$, we define norms
$\{|\cdot|_{S,i}\}_{i \in I}$ on $S \otimes_\bbQp T$ by
\[
|x|_{S,i}
 = \inf_{x = \sum_{k=1}^n s_k \otimes t_k}
   \left(\max_k |s_k|_S \cdot |t_k|_i \right).
\]
This makes $S \otimes_\bbQp T$ into a pre-Fr\'echet space, and we
declare $S \wh{\otimes} T$ to be its Fr\'echet completion, consisting
of equivalence classes of sequences that are simultaneously Cauchy
with respect to all the norms.  If $T$ is instead the direct limit of
the Fr\'echet spaces $\{T^j\}_{j \in J}$ (henceforth, we say {\it $T$
is LF}), we define $S \wh{\otimes} T$ to be the direct limit of the $S
\wh{\otimes} T^j$, each of the latter terms being defined above.

In particular, the above definitions apply to $T =
B^{\dag,r}_{\rig,K}$, which is Fr\'echet: there are norms $|\cdot|_s$
on $B^{\dag,r}_{\rig,K}$, for $0 < s \leq r$, corresponding to the
$\sup$ norms on the annuli $\ord_p(X) = s/e_K$, which can be described
easily in terms of the expansion of $f$ in $\pi_K$.  The definitions
also apply to $T = B^\dag_{\rig,K}$, which is the direct limit of the
$B^{\dag,r}_{\rig,K}$ for $r > 0$, and hence is LF.

We write $\Spm S$ for the collection of maximal ideals of $S$, and
when we have a label $x$ for an element $\fkm_x \in \Spm S$, we
abusively refer to $\fkm_x$ by $x$.  If $M$ is an $S$-module, we write
$M_x$ for $M \otimes_S S/\fkm_x$ throughout this section.

Applying $\otimes_{S \wh{\otimes} B^\dag_K} (S \wh{\otimes}
B^\dag_{\rig,K})$ to \cite[Th\'eor\`eme A]{BC}, as in \S6.2 of {\it
loc.\,cit.}, we see that one can canonically associate to $V$ a
locally free $S \wh{\otimes} B^\dag_{\rig,K}$-module
$\bfD^\dag_\rig(V)$ of rank equal to $\rank_S V$, equipped with
commuting, continuous, semilinear actions of $\vphi$ and $\Ga_K$, with
the property that $\bfD^\dag_\rig(V)_x$ is canonically isomorphic to
$\bfD^\dag_\rig(V_x)$ in $\bfM(\vphi,\Ga_K)_{/B^\dag_{\rig,K}}$ for
all $x \in \Spm S$.

Let us globalize this result.  Recall that our coefficient field for
Galois representations is $E$, a finite extension of $\bbQ_p$.  Let
$\scrX/E$ be a reduced, separated rigid analytic space with structure
sheaf $\calO_\scrX$.  The very notion of a rigid analytic space is
that $\scrX$ is built from its admissible affinoid subdomains $\scrU =
\Spm S$, so that a sheaf is determined by its restriction to an
admissible covering by admissible affinoid opens (for brevity, we will
call this a {\it good cover}), and a quasi-coherent sheaf is
determined by its {\it values} on such an open cover.  Note that an
affinoid algebra is naturally a $p$-adic Banach algebra; in fact, a
reduced affinoid algebra is a coefficient algebra.  A quasi-coherent
sheaf of $\calO_\scrX$-modules $M$ is said to be locally free of
finite rank (resp.\ locally free Banach, locally free Fr\'echet,
locally free LF) if $\scrX$ admits a good cover by opens $\scrU$ for
which $\Ga(\scrU,M) \approx \Ga(\scrU,\calO_\scrU) \wh{\otimes}
T_\scrU$, where $T_\scrU$ is a finite-dimensional $\bbQ_p$-vector
space (resp.\ a Banach space, a Fr\'echet space, an LF space).

For a commutative $\bbQ_p$-algebra $R$ that is finite-dimensional
(resp.\ Banach, Fr\'echet, LF) as a $\bbQ_p$-module, we define
$\calO_\scrX \wh{\otimes} R$ to be the locally free sheaf of
finite-dimensional (resp.\ Banach, Fr\'echet, LF)
$\calO_\scrX$-algebras with $T_\scrU = R$, as above, for every
affinoid subdomain $\scrU \subseteq \scrX$.  A locally free
$\calO_\scrX \wh{\otimes} R$-module of finite rank is a quasicoherent
sheaf $M$ of $\calO_\scrX \wh{\otimes} R$-modules on $\scrX$ such
that, for $\scrU$ ranging over some good cover of $\scrX$, each
$\Ga(\scrU,M)$ is free of finite rank over $\Ga(\scrU,\calO)
\wh{\otimes} R$.

In particular, we have defined the sheaf of rings $\calO_\scrE
\wh{\otimes} B^\dag_{\rig,K}$, and we have a notion of locally free
$\calO_\scrE \wh{\otimes} B^\dag_{\rig,K}$-module of finite rank.

Suppose we are given $\scrX/E$ as above, and a locally free
$\calO_\scrX$-module $\scrV$ of finite rank $d$ equipped with a
continuous, linear action $G_K$ of $\calO_\scrX$.  If $\scrU = \Spm S
\subset \scrX$ is any admissible affinoid neighborhood over which
$\scrV$ is free, then we can apply the theory of Berger--Colmez to
obtain $\bfD^\dag_\rig(\scrV|_\scrU)$.  One can check that the
constructions of Berger--Colmez are compatible with localization to
admissible affinoid subdomains of $\scrU$, so that the rule $\scrU
\mapsto \bfD^\dag_\rig(\scrV|_\scrU)$ on admissible affinoid
subdomains such that $\scrV|_\scrU$ is free determines a sheaf of
locally free $\calO_\scrX \wh{\otimes} B^\dag_{\rig,K}$-modules of
rank $d$ on $\scrX$, which we call $\scrD$.  It is equipped with
commuting, continuous, semilinear actions of $\vphi$ and $\Ga_K$, and
satisfies $\scrD_x \cong \bfD^\dag_\rig(\scrV_x)$ for all points $x
\in \scrX$.  (When writing ``$x \in \scrX$'', we always mean $x$ is a
physical point of $\scrX$, in the sense of Tate's rigid analytic
spaces.  The residue field $E(x)$ at $x$ is always finite over
$\bbQ_p$, since we have assumed $E/\bbQ_p$ finite.)

\subsection{Interpolation of the triangulordinary theory}
\label{sect-variational-tord}

It is our desire, in the future, to prove that a family consisting of
triangulordinary representations admits a corresponding family of
triangulordinary filtrations.  We state our goal in a preliminary
form, as the following conjecture.

Let $\scrX/E$ be a reduced, separated rigid analytic space, and let
$\scrD$ be a locally free sheaf of $(\vphi,\Ga_K)$-modules over
$\calO_\scrX \wh{\otimes} B^\dag_{\rig,K}$, of rank $d$.  Let $0 < c <
d$ be an integer.  Consider the functor that associates to an
$\scrX$-space $f\cn \scrU \to \scrX$ the collection of
$(\vphi,\Ga_K)$-stable $\calO_\scrU \wh{\otimes}
B^\dag_{\rig,K}$-local direct summands of $f^*\scrD$ rank $c$.

\begin{conj}\label{conj-kisin}
The functor described above is representable by a locally finite type
morphism $p_c \cn \scrX(c) \to \scrX$.  For each $x \in \scrX$, the
fiber $\scrX(c)_x$ is a finite union of quasiprojective flag varieties
over the residue field $E(x)$.
\end{conj}

\begin{rem}
This conjecture is inspired by results of Kisin, notably
\cite[Proposition 5.4]{Ki}.  The statements proved there involve a
number of technical hypotheses; thus, the above conjecture may require
some slight changes.  Under Kisin's hypotheses, we expect that his
methods may be translated into the $(\vphi,\Ga_K)$-module language to
establish the case where $c=d-1$.  In any case, to prove the
conjecture, it would suffice to work locally: to assume $\scrX$ is
affinoid, and to construct $\scrX(c)$ with the desired property for
maps $\scrU \to \scrX$ with $\scrU$ affinoid.
\end{rem}

The underlying {\it flag} of a filtration $F$ is simply the poset of
its constituents $F^i$, forgetting the indices.  By a {\it shape}
$\sigma$ (of rank $d$, and $r$ constituents), we mean a finite
sequence of dimensions $\{d = d_0 > d_1 > \cdots > d_{r+1} = 0\}$.  We
say that a flag $F$ of an object $D$ has shape $\sigma$ if $\rank D =
d$ and its constituents have dimensions given precisely by the
dimensions $d_i$ of $\sigma$; a filtration has shape $\sigma$ if its
underlying flag does.  We can consider the integer $c$ from above as
the shape with one constituent given by $\{d > c > 0\}$.  If $\sigma$
is an arbitrary shape of rank $d$ then, by inducting on the number of
constituents, Conjecture \ref{conj-kisin} implies the existence of a
locally finite type morphism $p_\sigma \cn \scrX(\sigma) \to \scrX$
classifying $(\vphi,\Ga_K)$-stable flags in $\scrD$ with shape
$\sigma$.  We write $\scrD(\sigma) := p_\sigma^*\scrD$, and denote by
$F(\sigma)$ the corresponding universal flag in $\scrD(\sigma)$ of
shape $\sigma$.

\begin{rem}
In the case $\sigma$ is the shape of a complete flag,
Bella\"iche--Chenevier give in \cite[Proposition 2.5.7]{BC2} an
affirmative answer to the Conjecture \ref{conj-kisin}, at least
infinitesimally locally: they prove the representability of the
related deformation problem.  They go on to undertake a considerably
detailed study of what amounts to the formal completion of
$\scrX(\sigma)$ at a crystalline point.
\end{rem}

We go on to explain how Conjecture \ref{conj-kisin}, in the more
general form just explained, should lead to triangulordinary
filtrations on the level of families.  First, we make precise what the
latter means.

We call $\scrD$ {\it pretriangulordinary} of shape $\sigma$ if there
is a Zariski dense subset $\scrX^\alg \subset \scrX$, all of whose
points $x \in \scrX^\alg$ satisfy the following property: the
$(\vphi,\Ga_K)$-module $\scrD_x$ is triangulordinary, with some (hence
every) triangulordinary filtration of shape $\sigma$.  We call $\scrD$
{\it triangulordinary} with respect to $F$, where $F^* \subseteq
\scrD$ is a decreasing, separated and exhaustive filtration by
$(\vphi,\Ga_K)$-stable $\calO_\scrX \wh{\otimes}
B^\dag_{\rig,K}$-local direct summands, if there is a Zariski dense
subset $\scrX^\alg \subset \scrX$ with the following property: for all
points $x \in \scrX^\alg$, the image $F_x$ has underlying flag equal
to the underlying flag of some triangulordinary filtration of
$\scrD_x$.  (For such $x$, the choices of indices making $F_x$ into
the triangulordinary filtration are then uniquely determined by the
Hodge--Tate weights.)  Clearly, if $\scrD$ is triangulordinary with
respect to $F$, and $F$ has shape $\sigma$, then $\scrD$ is
pretriangulordinary of shape $\sigma$.

When $\scrD = \bfD^\dag_\rig(\scrV)$, we say that $\scrV$ is
pretriangulordinary of shape $\sigma$ (resp.\ triangulordinary with
respect to $F$) if the said condition holds for $\scrD$.

Suppose $\scrD$ is pretriangulordinary of shape $\sigma$, and let
$p_\sigma \cn \scrX(\sigma) \to \scrX$ classify $(\vphi,\Ga_K)$-stable
flags of shape $\sigma$, as above.  We let $\scrX_\tord^\alg =
\scrX_\tord^\alg(\sigma)$ be the set of $x \in p_\sigma\inv
\scrX^\alg$ such that $F(\sigma)_x$ is the underlying flag of a
triangulordinary filtration on $\scrD_x$, and we let $\scrX_\tord$ be
the Zariski closure of $\scrX_\tord^\alg$ inside $\scrX$.  We write
$\scrD_\tord$ (resp.\ $F_\tord$) for the restriction of
$\scrD(\sigma)$ (resp.\ $F(\sigma)$) to $\scrX_\tord$.  By
construction, $\scrD_\tord$ is a triangulordinary family over
$\scrX_\tord$ of shape $\sigma$ with respect to any choice of indices
making the flag $F_\tord$ into a filtration.  Moreover, by the
hypothesis that $\scrX$ is pretriangulordinary, the restriction of
$p_\sigma$ is a surjection $\scrX_\tord^\alg \twoheadrightarrow
\scrX^\alg$, so that $\scrX_\tord$ is rather substantial in comparison
with $\scrX$.

\begin{rem}\label{rem-indices}
It is not clear from the above discussion whether the construction
singles out a choice of indices for the triangulordinary flag
$F_\tord$ on $\scrD_\tord$.  We would hope for the ``most
appropriate'' choice of indexing to have the following property: for
all $x \in \scrX^\alg$, the constituent $F^1$ has image in $\scrD_x$
equal to the $F^1$ of some triangulordinary filtration on $\scrD_x$
satisfying the hypotheses of Theorem \ref{thm-local-main}(2).  Thus,
the choice of indexing {\it does} affect which Selmer group is
obtained from the definition given in the next section.

Whether a ``most appropriate'' indexing exists, and (if it exists)
which indexing it is, are sensitive to the specification of
$\scrX^\alg \subset \scrX$.  See Remark \ref{rem-choosing-indices} for
an example.  (Although we have not stressed this, the construction of
$\scrX_\tord$ itself depends on $\scrX^\alg$.)
\end{rem}

\begin{exmp}\label{exmp-univ-deformation}
Assume the notation at the end of \S\ref{sect-definitions}, with
$K=\bbQ$ and $p>2$.  Fix a $2$-dimensional, irreducible, odd
representation $\ov{\rho}$ of $G_{\bbQ,S}$ with values in the residue
field $k_E$ of $E$.  We take for $\scrX$ the generic fiber of $\Spf
R^\text{univ}_S(\ov{\rho})$, where $R^\text{univ}_S(\ov{\rho})$ is the
universal $\calO_E$-valued deformation ring with ``unramified'' local
conditions away from $S$, and no conditions at $S$.  We take for
$\scrV$ the universal representation on this space, and $\scrD =
\bfD^\dag_\rig(\scrV|_{G_p})$.  Since the set $\scrX^\alg$ of points
$x \in \scrX$ for which $\scrV_x|_{G_p}$ is semistable with distinct
Frobenius eigenvalues is Zariski dense, one can deduce that $\scrD$ is
pretriangulordinary of shape $\sigma = \{2 > 1 > 0\}$.

Granting Conjecture \ref{conj-kisin}, we expect that
$\scrX_\tord(\sigma)$ is none other than the eigensurface of
Coleman--Mazur (discussed towards the end of
\S\ref{sect-expectations}).  Its restriction to the subspace of
$\scrX$ having a vanishing Sen weight is expected to be the
eigencurve, obtained as the resolution of the infinite fern of
Gouv\^ea--Mazur \cite{GM} at its double points.  Thus, our setup ought
to give a clean realization of Kisin's hope of constructing general
eigenvarieties purely Galois-theoretically.
\end{exmp}

\subsection{Selmer groups via variation}
\label{sect-variational-selmer}

In order to define Selmer groups of families of Galois
representations, we need to give a meaning to the Galois cohomology of
a family.  Let $G$ be a profinite group acting continuously on a
locally free module $\scrV$ of finite rank over a rigid analytic space
$\scrX$.  We let
\[
H^i(G,M) := \Ext^i_{\calO_\scrX[G]}({\bf 1},M),
\]
the Yoneda group in the category of locally free $\calO_\scrX$-modules
with continuous $G$-actions.  As is customary, when $G = G_K$ is the
absolute Galois group of a field $K$, we write $H^i(K,M)$ for
$H^i(G_K,M)$.

We now resume the notation at the end of \S\ref{sect-definitions}.
Namely, $K/\bbQ$ is a finite extension, $S$ is a finite set of places,
and we have algebraic closures $\ov{K}_v$ containing $K_S$ and maps
$G_v \to G_{K,S}$, for $v \in S$.

We let $\scrX/E$ be as in the preceding section, and let $\scrV$ be a
locally free sheaf on $\scrX$ of finite rank, equipped with a
continuous, $\calO_\scrX$-linear $G_{K,S}$-action.  We assume, for
each place $v$ of $K$ with $v \mid p$, that $\scrV|_{G_v}$ is
triangulordinary with respect to some filtration $F_v^* \subseteq
\scrD_v := \bfD^\dag_\rig(\scrV|_{G_v})$, in the sense described in
\S\ref{sect-variational-tord}.  (Whether or not Conjecture
\ref{conj-kisin} holds, we assume here that we are simply given the
$F_v^*$.)  For such $v$ we define the local condition at $v$ to be
\begin{align}\label{eqn-selmer1}
H^1_\tord(K_v,\scrV) = \ker \Bigg[
H^1(K_v,\scrV)
&=
\Ext^1_{\calO_\scrX[G_v]}({\bf 1},\scrV) \nonumber \\
&\stackrel{\star}{\to}
\Ext^1_{\vphi,\Ga_{K_v}/\calO_\scrX \wh{\otimes}
  B^\dag_{\rig,K_v}}({\bf 1}, \scrD_v) \\
&\to
\Ext^1_{\vphi,\Ga_\wh{K_v^\unr}/\calO_\scrX \wh{\otimes}
  B^\dag_{\rig,\wh{K_v^\unr}}}({\bf 1}, (\scrD_v/F_v^1)_\wh{K_v^\unr})
\Bigg]. \nonumber
\end{align}
Assuming that, for $x \in \scrX^\alg$, the specialization $(F_v^1)_x$
is the $F^1$ of a triangulordinary filtration on $\scrV_x$, the above
definition provides an interpolation over all of $\scrX$ of the ``g+''
Bloch--Kato local conditions at the points of $\scrX^\alg$.  By
Proposition \ref{ppn-local-BK}, at all such points this agrees with
the ``g'' local condition, and at most such points this agrees with
the ``f'' condition.  Thus, it is reasonable to define the Selmer
group of $\scrV$ over $\scrX$ to be
\begin{multline}\label{eqn-selmer2}
H^1_\tord(K,\scrV) = \ker \Bigg[
H^1(G_{K,S},\scrV) \\
\to
\bigoplus_{v \in S,\ v \nmid p}
\frac{\displaystyle H^1(K_v,\scrV)}{\displaystyle  H^1_\unr(K_v,\scrV)}
\oplus \bigoplus_{v \mid p}
\frac{\displaystyle H^1(K_v,\scrV)}{\displaystyle  H^1_\tord(K_v,\scrV)}
\Bigg].
\end{multline}

\begin{rem}
The map labeled $\star$ in Equation \ref{eqn-selmer1} is not known to
be an isomorphism.  In fact, in contrast to the situation of Theorem
\ref{thm-phigamma-equiv}, the map $\bfD^\dag_\rig$ from families of
Galois representations to \'etale families of $(\vphi,\Ga_K)$-modules
is {\it not} an equivalence of categories.  Chenevier has given the
following counterexample.  Denote by $K\bra{\ul{T}}$ the Tate algebra
over $K$ in the variables $\ul{T}$.  Let $D = \bbQ_p\bra{T,T\inv}
\wh{\otimes}_{\bbQ_p} B^\dag_{\rig,K} \cdot e$, so that $D$ has rank
$1$ with basis element $e$, with actions given by $\Ga_K \cdot e = e$
and $\vphi(e) = Te$.  There is no Galois representation $V$ over
$\bbQ_p\langle T,T\inv \rangle$ for which $\bfD^\dag_\rig(V) = D$.

We ask whether it is still the case that $\star$ is an isomorphism: if
two $(\vphi,\Ga_K)$-modules over $S \wh{\otimes} B^\dag_{\rig,K}$ come
from $S$-representations of $G_K$, does every extension between them
come from an $S$-representation?  In any case, for every $x \in
\scrE^0$ the diagram
\begin{equation}\label{eqn-specialize-selmer}\begin{array}{ccc}
H^1(K_v,\scrV) & \to & \Ext^1({\bf 1},\scrD) \\
\downarrow & & \downarrow \\
H^1(K_v,\scrV_x) & \stackrel{\sim}{\to} & \Ext^1({\bf 1},\scrD_x)
\end{array}\end{equation}
commutes, so we at least know that we are imposing the correct local
condition, specialization-by-specialization, everywhere we are able
to.
\end{rem}

We also obtain notions of Selmer groups $H^1_\tord(K,\scrV_x)$ for
{\it all} specializations $\scrV_x$ of $\scrV$ with $x \in \scrX$:
namely, we define $F^1_{x,v}$ to be $(F^1_v)_x$, and add subscripts
$x$ everywhere in Equations (\ref{eqn-selmer1}--\ref{eqn-selmer2}).
It follows from the commutativity of Diagram
\ref{eqn-specialize-selmer} above that there is a natural
specialization map
\[
H^1_\tord(K,\scrV)_x \to H^1_\tord(K,\scrV_x).
\]
for each $x \in \scrX$.  Perhaps the most important open question in
our program is whether an analogue of Mazur's control theorem holds:
when can we bound the kernel and cokernel of the above map?  Can this
bounding be achieved, uniformly for $x$ varying through a substantial
subset of $\scrX$?  Although we strongly desire to check this in a
concrete setting, at present we cannot handle any particular
nonordinary case.

We go on now to discuss in detail our model example: the eigencurve of
Coleman--Mazur.

\subsection{Review of the eigencurve}
\label{sect-review-eigencurve}

We continue with the notations of the end of \S\ref{sect-definitions}.
We assume $K=\bbQ$, and we fix positive integer $N$ not divisible by
$p$, which we call the tame level.  We take for $S$ the set of primes
dividing $p$ and $N$, together with the place $\infty$.

By the weight space $\scrW$ we mean the rigid analytic space over
$\bbQ_p$ arising as the generic fiber of $\Spf \bbZ_p[\![\bbZ_p^\times
\times (\bbZ/N)^\times]\!]$.  Its points $\scrW(R)$ correspond to
continuous characters of the form $\bbZ_p^\times \times
(\bbZ/N)^\times \to R^\times$.  By class field theory, one can
consider $\scrW$ as being equipped with a free rank $1$ bundle $\scrT$
on which $G_{\bbQ,S}$ acts through its universal character.  We let
$\scrW^\alg$ consist of those points $w \in \scrW$ corresponding to
characters having the form $a \mapsto a^{k_w}$ on some open subgroup
of $\bbZ_p^\times \subseteq \bbZ_p^\times \times (\bbZ/N)^\times$,
with $k_w$ an integer.  Clearly, $\scrW$ is triangulordinary of shape
$\{1 > 0\}$.

The {\it eigencurve} $\scrE = \scrE_{p,N}$, defined in \cite{CM} in
the case $p > 2$ and $N = 1$, and extended to general $p$ and $N$ by a
variety of authors, is the following object.  It is a rigid analytic
space over $\bbQ_p$, locally-on-the-base finite over $\scrW \times
(\scrB_1(0) \bs \{0\})$, and locally-in-the-domain finite flat over
$\scrW$.  Here, $\scrB_1(0)$ is the closed unit disk around the
origin, $\scrB_1(0) = \Spm \bbQ_p\bra{U_p}$.  The map to $\scrW$ is
called the weight (or, more precisely, weight-nebentypus), the map to
$\scrB_1(0)$ is called the $U_p$-eigenvalue, and the latter's
composite with the valuation map is called the slope.  Finally,
$\scrE$ parameterizes a universal rigid analytic family of pairs
$(f,\alpha)$ with $f$ a $p$-adic overconvergent elliptic modular
eigenform of tame level $N$ and $\alpha$ a nonzero (``finite slope'')
$U_p$-eigenvalue of $f$.  We let $\scrE^\alg$ be the collection of
points $x \in \scrE$ corresponding to pairs $(f_x,\alpha_x)$ with
$f_x$ classical of weight $k_x \geq 2$, and with $f_x$ at worst
semistable at $p$.

We remove two types of bad points on $\scrE$.  Namely, we say that $x
\in \scrE$ has {\it critical slope} if its weight $w$ lies in
$\scrW^\alg$ with $k_w$ as above, $k_w \geq 2$, and $k_w-1$ is the
slope of $x$.  (This agrees with the terminology of
\S\ref{sect-local-ex}, except that it also includes some nonclassical
$x$.)  We say that $x \in \scrE^\alg$ has {\it does not have distinct
eigenvalues} if $\alpha_x$ is a double root of the $p$-Hecke
polynomial of (the newform associated to) $f_x$.  Write $\scrE^0
\subset \scrE$ for the complement of the critical-slope and
not-distinct-eigenvalue loci, and $\scrE^{0,\alg} = \scrE^0 \cap
\scrE^\alg$.  (One can do slightly better as regards critical slope,
and instead look at the complement of the points in the image of the
$\theta^{k-1}$-map for each $k \geq 2$.)

The constructions of $\scrE$ give rise to a locally free rank $2$
bundle $\scrV^0$ over $\scrE^0$, equipped with a continuous,
$\calO_{\scrE^0}$-linear action of $G_{\bbQ,S}$.  This representation
has the property that, for any $x \in \scrE^{0,\alg}$, the fiber
$\scrV^0_x$ is isomorphic to the Galois representation
$V_{f_x}^\text{hom}$ associated to $f_x$ by Deligne.

\begin{rem}
The reader will note that the Galois representation
$V_{f_x}^\text{hom}$ is defined for every $x \in \scrE$ with $f_x$
classical of weight $k_x \geq 2$.  Our restriction to semistable
points is because they would require modifying the triangulordinary
theory to handle representations that become semistable over an
abelian extension.  We exclude $\scrE^\alg \bs \scrE^{0,\alg}$ from
our consideration only because Kisin does so in \cite{Ki}; we have not
tested whether our theory should make sense at these points.
\end{rem}

We write $\scrD^0 = \bfD^\dag_\rig(\scrV^0|_{G_p})$, so that for $x
\in \scrE^\alg$ one has
\[
\scrD^0_x \cong \bfD^\dag_\rig(\scrV^0_x|_{G_p}) \cong
\bfD^\dag_\rig(V_{f_x}^\text{hom}|_{G_p}).
\]
For every $x \in \scrE^{0,\alg}$ the representation
$V_{f_x}^\text{hom}|_{G_p}$ is semistable, and $\alpha_x$ is a
$\vphi\inv$-eigenvalue $\nu$ on $\bfD_\st(V_{f_x}^\text{hom}|_{G_p})$.
Since we have removed the not-distinct-eigenvalue locus from
$\scrE^0$, $\bfD_\st(V_{f_x}^\text{hom}|_{G_p})$ has distinct
$\vphi$-eigenvalues, so the $\nu$-eigenspace gives rise to a canonical
triangulordinary filtration
\begin{equation}\label{eqn-filtration-x}
\bfD^\dag_\rig(V_{f_x}^\text{hom}|_{G_p}) = F^0_k \supsetneq F^1_x =
F^{k_w-1}_x \supsetneq F^{k_w}_x = 0,
\end{equation}
as in \S\ref{exmp-MFs}, where $w$ is the weight of $x$.  Therefore,
$\scrD^0$ is pretriangulordinary of shape $\{2 > 1 > 0\}$.  From now
on, we denote this particular shape by $\sigma$.

\subsection{Expectations for the eigencurve}
\label{sect-expectations}

We expect that $\scrD^0$ is triangulordinary of shape $\sigma$ in the
following precise sense:

\begin{conj}\label{conj-eigencurve}
There exists a unique filtration
\[
\scrD^0 \supsetneq F^1 \supsetneq 0
\]
by a $(\vphi,\Ga_\bbQp)$-stable locally $\calO_{\scrE^0} \wh{\otimes}
B^\dag_{\rig,\bbQp}$-direct summand $F^1$ of rank $1$ with the
property that, for each $x \in \scrE^{0,\alg}$, $(F^1)_x = F^1_x$
under the identification of $\scrD^0_x \cong
\bfD^\dag_\rig(\scrV_x|_{G_p}) \cong
\bfD^\dag_\rig(V^\text{hom}_{f_x}|_{G_p})$.
\end{conj}

An equivalent way of formulating the conjecture is as follows.  Let
$\mathscr{F} \subset \scrD^0$ be the subsheaf defined by
\[
\mathscr{F} := \bigcap_{x \in \scrE^{0,\alg}}
 \ker\left[ \scrD^0 \to \scrD^0_x/F_x^1 \right].
\]
Then we may phrase Conjecture \ref{conj-eigencurve} as asserting the
existence of a unique $(\vphi,\Ga_\bbQp)$-stable locally
$\calO_{\scrE^0} \wh{\otimes} B^\dag_{\rig,\bbQp}$-direct summand
$F^1$ of $\scrD^0$ of rank $1$ contained in $\mathscr{F}$.  If so,
then the specialization maps $\scrD^0_x \cong
\bfB^\dag_\rig(\scrV_x|_{G_p})$ automatically identify $(F^1)_x \cong
F_x^1$ for all $x \in \scrE^{0,\alg}$.

Suppose that Conjecture \ref{conj-kisin} holds with $d=2$ and $c=1$,
so that the morphism $p_\sigma \cn \scrE^0(\sigma) \to \scrE^0$
exists.  Then Conjecture \ref{conj-eigencurve} says that the
assignment
\begin{eqnarray*}
\scrE^{0,\alg} & \to & \scrE^{0,\alg}_\tord \subset \scrE^0(\sigma) \\
x & \mapsto & (x,F_x^1)
\end{eqnarray*}
extends uniquely to a section $\scrE^0 \to \scrE^0(\sigma)$ of
$p_\sigma$.  In fact, any such section must have image in
$\scrE^0_\tord$, and we expect that this is the only section of
$p_\sigma|_{\scrE^0_\tord}$.  We envision that $\scrE^0_\tord$ can be
divided into two parts: a component mapping isomorphically onto
$\scrE^0$, and a disjoint union of points, one lying over each $x \in
\scrE^{0,\alg}$ with $f_x$ crystalline at $p$, corresponding to the
``evil twin'' of $(f_x,\alpha_x)$ (as in \cite{GM}).

The reader is advised to take note of the difference between the above
picture and Example \ref{exmp-univ-deformation}.

Conjecture \ref{conj-eigencurve} gives us a definition of the Selmer
group over the eigencurve, as well as Selmer groups for all
finite-slope overconvergent modular eigenforms, as in Equation
\ref{eqn-selmer2}.

\begin{rem}
An arbitrary (especially, nonclassical) $x \in \scrE^0$ is known to be
trianguline by work of Kisin and Colmez.  Let $\bbQ_p(x)$ denote the
residue field of $\scrE^0$ at $x$, and $(f_x,\alpha_x)$ the
corresponding $\bbQ_p(x)$-valued overconvergent eigenform and
$U_p$-eigenvalue.  Then \cite[Theorem 6.3]{Ki} shows the existence of
a nonzero, $G_p$-equivariant map $V_{f_x}^\text{hom} \to (B_\crys^+
\otimes_\bbQp \bbQ_p(x))^{\vphi=\alpha_x}$.  This is equivalent to a
nonzero vector in
$\bfD_\crys(V_{f_x}^\text{coh}|_{G_p})^{\vphi=\alpha_x}$.  Then
\cite[Proposition 5.3]{C} implies that $V_{f_x}^\text{coh}$ is
trianguline at $p$, and hence so is $V_{f_x}^\text{hom}$.  The
trianguline subspace $F$ inside
$\bfD^\dag_\rig(V_{f_x}^\text{hom}|_{G_p})$ ought to coincide with the
putative triangulordinary filtration $F_x^1$ described above when
Conjecture \ref{conj-eigencurve} holds.  In any case, using $F$ in
place of $F_x^1$, we obtain a definition of a local condition and
Selmer group without assuming any conjecture.

We remind the reader that, although the work of Kisin and Colmez gives
us trianguline filtrations at every point, they do not directly give
us a filtration on the family.
\end{rem}

A related example is the cyclotomic deformation of $\scrV^0$.  This is
the bundle $\wt{\scrV}^0$ over the {\it eigensurface} $\wt{\scrE}^0 =
\scrE^0 \times \scrW$ determined by $p_1^* \scrV^0
\otimes_{\wt{\scrE}^0} p_2^* \scrT$, where the $p_i$ are the
projections of $\wt{\scrE}^0$ onto the respective factors.  Letting
$\wt{\scrE}^{0,\alg} = \scrE^{0,\alg} \times \scrW^\alg$, we see that
$\wt{\scrV}^0$ is pretriangulordinary of shape $\sigma = \{2 > 1 >
0\}$.  Assuming Conjecture \ref{conj-eigencurve} and setting $\wt{F}^1
= p_1^*F^1 \otimes_{\wt{\scrE}^0} p_2^*\scrT$, we see that
$\wt{\scrV}^0$ is also triangulordinary of shape $\sigma$.

\begin{rem}\label{rem-choosing-indices}
In the case of the eigencurve $\scrE^0$, every $\scrV_x|_{G_p}$ with
$x \in \scrE^{0,\alg}$ has Hodge--Tate weights $0$ and $k-1 > 0$.
Therefore, as seen in Equation \ref{eqn-filtration-x}, the
triangulordinary filtrations all have $F_x^1$ equal to their rank-$1$
constituent.  When assigning indices to the putative triangulordinary
flag $F$ on $\scrD^0$ given by Conjecture \ref{conj-eigencurve}, this
fact forces us to take $F^1$ to be the rank-$1$ constituent.  In other
words, the Galois theory provided us with a natural choice of
filtration indexing, and, by consequence, a natural choice of Selmer
group.

Consider the universal character $\scrT$ of $G_{\bbQ,S}$ over weight
space $\scrW$.  The unique Hodge--Tate weight of $w \in \scrW^\alg$ is
$k_w$, and these integers vary without bound.  Thus there is no ``most
appropriate'' index at which to situate the jump in the
triangulordinary filtration, compatibly over all of $\scrW^\alg$ as
defined in \S\ref{sect-review-eigencurve}.  Another viewpoint is that
$\scrT$ is the cyclotomic {\it deformation} of the trivial character
$\chi_\triv$, and hence its triangulordinary filtration should be
chosen to deform the natural one for $\chi_\triv$.  Since $\chi_\triv$
has Hodge--Tate weight $0$, this means taking $\Gr^0 \neq 0$, and, in
particular, $F^1 = 0$.  Another way of achieving this would be to
reduce $\scrW^\alg$ to its subset consisting of those $w$ with $k_w =
0$ (which is still Zariski dense).  A third option is to note that
$\scrT$ is also the cyclotomic deformation of the cyclotomic character
$\chi_\cycl$, which corresponds to replacing $\scrW^\alg$ with the
subset defined by $k_w = 1$, and which suggests taking $F^1 = \scrT$.
Thus, depending on the choice of $\scrW^\alg$, the most appropriate
indexing of the triangulordinary flag either does not exist, has
$F^1=0$, or has $F^1=\scrT$.  The latter two possibilities give two
different Selmer local conditions at $p$ (respectively, they are the
unramified and empty conditions).

The ambiguity described above passes on to $\wt{\scrV}^0$: at a point
$(x,w)$, where $x$ has weight $k_x$ and $w$ has weight $k_w$, the
Hodge--Tate weights of $\wt{\scrV}^0_{(x,w)}$ are $k_w$ and $k_w+k_x-1
> k_w$, which vary roughly independently; thus $\wt{\scrE}^{0,\alg}$
does not admit a most appropriate choice of indices.  Since we view
the eigensurface as a the cyclotomic deformation of the eigencurve, we
expect that considering $\scrT$ as the cyclotomic deformation of the
{\it trivial} character is most appropriate (in this particular
setting).  This means reducing $\wt{\scrE}^{0,\alg}$ to the subset
defined by $k_w = 0$, and taking for $\wt{F}^1$ the rank-$1$
constituent of the triangulordinary flag, which is given by $p_1^*F^1
\otimes_{\wt{\scrE}^0} p_2^*\scrT$.
\end{rem}

Since the reader is likely to be aware of the goals of Iwasawa theory,
we conclude by saying that we expect the Selmer group
$H^1_\tord(\bbQ,\scrE^0)$ (resp.\ $H^1_\tord(\bbQ,\wt{\scrE}^0)$) to
be related to the analytic standard $p$-adic $L$-function varying
along the eigencurve (resp.\ eigensurface).  But, the Selmer groups
being highly non-integral (and likely non-torsion), and the $p$-adic
$L$-functions being unbounded, the precise means by which these ought
to be related related is far from clear.



\bibliography{tord}

\end{document}